\documentclass[oneside,reqno,english]{amsart}
\usepackage[T1]{fontenc}
\usepackage[utf8]{inputenc}
\usepackage{xcolor}
\usepackage{babel}
\usepackage{prettyref}
\usepackage{amstext}
\usepackage{amsthm}
\usepackage{amssymb}
\usepackage[pdfusetitle,
 bookmarks=true,bookmarksnumbered=false,bookmarksopen=false,
 breaklinks=false,pdfborder={0 0 0},pdfborderstyle={},backref=false,colorlinks=false]
 {hyperref}
\hypersetup{
 colorlinks=true,citecolor=blue,linkcolor=blue,linktocpage=true}

\makeatletter

\newcommand*\LyXZeroWidthSpace{\hspace{0pt}}

\numberwithin{equation}{section}
\numberwithin{figure}{section}

\usepackage{prettyref}

\newrefformat{cor}{Corollary~\ref{#1}}
\newrefformat{subsec}{Section~\ref{#1}}
\newrefformat{lem}{Lemma~\ref{#1}}
\newrefformat{thm}{Theorem~\ref{#1}}
\newrefformat{sec}{Section~\ref{#1}}
\newrefformat{chap}{Chapter~\ref{#1}}
\newrefformat{prop}{Proposition~\ref{#1}}
\newrefformat{exa}{Example~\ref{#1}}
\newrefformat{tab}{Table~\ref{#1}}
\newrefformat{rem}{Remark~\ref{#1}}
\newrefformat{def}{Definition~\ref{#1}}
\newrefformat{fig}{Figure~\ref{#1}}
\newrefformat{claim}{Claim~\ref{#1}}
\newrefformat{assu}{Assumption~\ref{#1}}

\makeatother

\theoremstyle{plain}
\newtheorem{thm}{\protect\theoremname}[section]
\newtheorem{prop}[thm]{\protect\propositionname}
\theoremstyle{definition}
\newtheorem{defn}[thm]{\protect\definitionname}
\newtheorem{example}[thm]{\protect\examplename}
\theoremstyle{remark}
\newtheorem*{rem*}{\protect\remarkname}
\theoremstyle{plain}
\newtheorem{lem}[thm]{\protect\lemmaname}
\newtheorem{cor}[thm]{\protect\corollaryname}
\theoremstyle{remark}
\newtheorem{rem}[thm]{\protect\remarkname}
\theoremstyle{plain}
\newtheorem{assumption}[thm]{\protect\assumptionname}
\newtheorem*{lyxalgorithm*}{\protect\algorithmname}
\newtheorem{lyxalgorithm}[thm]{\protect\algorithmname}
\providecommand{\algorithmname}{Algorithm}
\providecommand{\assumptionname}{Assumption}
\providecommand{\corollaryname}{Corollary}
\providecommand{\definitionname}{Definition}
\providecommand{\examplename}{Example}
\providecommand{\lemmaname}{Lemma}
\providecommand{\propositionname}{Proposition}
\providecommand{\remarkname}{Remark}
\providecommand{\theoremname}{Theorem}

\begin{document}
\title[Telescoping Residuals in Operator Products and Kernels]{Use of operator defect identities in multi-channel signal plus residual-analysis
via iterated products and telescoping energy-residuals: Applications
to kernels in machine learning}
\begin{abstract}
We present a new operator theoretic framework for analysis of complex
systems with intrinsic subdivisions into components, taking the form
of “residuals” in general, and \textquotedbl telescoping energy residuals\textquotedbl{}
in particular.  We prove new results which yield admissibility/effectiveness,
and new a priori bounds on energy residuals.  Applications include
infinite-dimensional Kaczmarz theory for $\lambda_{n}$-relaxed variants,
and $\lambda_{n}$-effectiveness. And we give applications of our
framework to generalized machine learning algorithms, greedy Kernel
Principal Component Analysis (KPCA), proving explicit convergence
results, residual energy decomposition, and criteria for stability
under noise.
\end{abstract}

\author{Palle E.T. Jorgensen}
\address{(Palle E.T. Jorgensen) Department of Mathematics, The University of
Iowa, Iowa City, IA 52242-1419, U.S.A.}
\email{palle-jorgensen@uiowa.edu}
\author{Myung-Sin Song}
\address{(Myung-Sin Song) Department of Mathematics and Statistics, Southern
Illinois University Edwardsville, Edwardsville, IL 62026, USA}
\email{msong@siue.edu}
\author{James Tian}
\address{(James Tian) Mathematical Reviews, 535 W. William St, Suite 210, Ann
Arbor, MI 48103, USA}
\email{james.ftian@gmail.com}

\maketitle
\tableofcontents{}

\section{Introduction}\label{sec:1}

Recent papers dealing with recursive algorithms build on use of limits
of finite operator products (where the factors are projections or
more general operators), and with a single parameter $\lambda$ serving
to quantify regret estimates. Motivated by applications, we turn here
to a wider setting. In particular, rather than the use a single value
of $\lambda$, we identify here a wider framework which makes use
instead of a sequence of numbers $\lambda_{n}$, one for each operator
factor in the recursive construction. We further merge this new setting
into a framework kernel analysis, as used in new directions in machine
learning (ML). And we add two new applications to ML.

The main results in our paper are the following two: \prettyref{thm:2-4}
which gives a new telescoping identity, and a summability criterion,
both in a new operator theoretic framework. And \prettyref{thm:3-1},
based on kernel analysis, where we prove an exact energy balance for
the $\lambda_{n}$ dynamics, as well as an explicit error bound on
the summation of the squared residuals at sample points, serving to
guarantee effectiveness.

The paper is organized as follows. We introduce the operator-theoretic
framework of multi-channel defect splitting and establish a fundamental
telescoping identity for products of operators in \prettyref{sec:2}.
Exact residual energy decompositions, summability criteria, and convergence
results that form the foundation for the subsequent analysis are discussed.

We transport the framework from \prettyref{sec:2} in \prettyref{sec:3}
to RKHSs and discuss kernel interpolation and related online learning
dynamics. We prove an exact telescoping energy balance, yielding explicit
convergence guarantees and effectiveness criteria expressed in terms
of rank-one projections associated with sampling points.

\prettyref{sec:4} extends the deterministic framework to stochastic
settings, including random sampling and dropout mechanisms. Stability
and robustness under noise are analyzed, and quantitative bias--variance
tradeoffs are derived under mild step-size conditions.

We then develop a tree-based reparametrization of kernels using the
multichannel splitting structure in \prettyref{sec:5}. This leads
to multiscale feature representations and truncated kernels. The approximation
error is measured exactly by discarded residual energy, with sharp
bounds at both the kernel and Gram-matrix levels.

Building on this framework, \prettyref{sec:6} introduces a greedy
refinement strategy for kernel approximation and compression. Here,
kernel components are selected adaptively based on empirical energy
contributions. Thus, resulting in monotone convergence to the original
kernel and providing practical stopping criteria.

Lastly, \prettyref{sec:7} connects the proposed framework to greedy
kernel principal component analysis, highlighting implications, and
the advantages of the telescoping residual framework---exact energy
decomposition, convergence guarantees, and noise stability.

\section{Multi-channel defect splitting}\label{sec:2}

Our aim is to present a mathematical framework which arises in a variety
of applied settings involving complex systems subdivided into components,
here referred to as ``residuals'' or differences in energy at different
levels. In the technical literature, the term \textquotedbl telescoping
energy residuals\textquotedbl{} refers to a concept used in engineering
contexts, in mathematical physics, structural engineering, and in
computational fluid dynamics (CFD). In these diverse applications,
it typically describes a method where large systems are broken down
into components such that the ``residuals,'' or differences in energy
levels, can be summed via an algorithmic, often \textquotedbl telescoping\textquotedbl{}
manner. A ``residual'' typically reflects a difference between an
observed, or actual value, and a value predicted by the model. Specifically,
``energy residuals'' refer to remaining energy in a system or simulation,
which is unaccounted-for.

We begin with a mechanism common to a broad class of iterative splitting
algorithms: the defect identity associated with a contraction. This
identity yields a canonical two-channel ``signal plus residual''
decomposition at each step, and when iterated it produces a telescoping
energy-residual formula. Classical Kaczmarz-type methods, as well
as $\lambda$-relaxed projection schemes, fit naturally into this
contraction framework. In those settings one typically imposes an
admissibility/effectiveness hypothesis to ensure that the residual
channel vanishes asymptotically (or equivalently, that the energy
residual converges to the correct limit); see, for example, \cite{MR1881441,MR4126821}
for infinite-dimensional Kaczmarz theory and \cite{jeong2025} for
$\lambda$-relaxed variants and $\lambda$-effectiveness.

For special cases of related uses of infinite operator products in
harmonic analysis, effective systems, and in generalized machine learning
algorithms, we wish to call attention to the following recent papers
\cite{MR4941845,MR4957888,MR4887444,MR4867336,MR4835160,MR4676388,MR4619155}. 

\subsection{Defect decomposition}

For our present setting it will be convenient for us to work in the
general context of Hilbert space, denoted $H$, here given with specified
norm and inner product. The bounded linear operators in $H$ will
be denoted $B(H)$. But for estimation in algorithms, to be discussed
throughout the paper, it will be useful for us to specialize to the
case of contractive operators. In \prettyref{sec:3} below we shall
specialize to the case when $H$ is taken to be a reproducing kernel
Hilbert space, RKHS.

Fix a contraction $A\in B\left(H\right)$ and define its defect operator
\[
D_{A}=\left(I-A^{*}A\right)^{1/2}\in B\left(H\right)_{+}.
\]
Consider the map 
\[
W_{A}:H\to H\oplus H,\qquad W_{A}x=\left(Ax,D_{A}x\right).
\]
Then $W_{A}$ is an isometry. Indeed, 
\begin{align*}
\left\Vert W_{A}x\right\Vert ^{2} & =\left\Vert Ax\right\Vert ^{2}+\left\Vert D_{A}x\right\Vert ^{2}\\
 & =\left\langle x,A^{*}Ax\right\rangle +\left\langle x,\left(I-A^{*}A\right)x\right\rangle =\left\Vert x\right\Vert ^{2}.
\end{align*}
Equivalently, 
\[
\left\Vert x\right\Vert ^{2}=\left\Vert Ax\right\Vert ^{2}+\left\Vert D_{A}x\right\Vert ^{2},\qquad x\in H,
\]
and in operator form, 
\[
I=A^{*}A+D^{2}_{A}.
\]

Iterating this identity is straightforward, but it is the cleanest
way to see the origin of the standard ``energy residual'' and its
associated residual channel.
\begin{prop}
\label{prop:2-1} Let $A_{1},A_{2},\dots$ be contractions on $H$.
Set $T_{0}=I$ and $T_{n}=A_{n}A_{n-1}\cdots A_{1}$ for $n\ge1$.
Then for every $x\in H$ and every $N\ge1$, 
\[
\left\Vert x\right\Vert ^{2}=\left\Vert T_{N}x\right\Vert ^{2}+\sum^{N}_{n=1}\left\Vert D_{A_{n}}T_{n-1}x\right\Vert ^{2}.
\]
Equivalently, for every $N\ge1$, 
\[
I-T^{*}_{N}T_{N}=\sum^{N}_{n=1}T^{*}_{n-1}D^{2}_{A_{n}}T_{n-1},
\]
where the sum is finite and hence unambiguous in $B\left(H\right)$. 
\end{prop}

\begin{proof}
Apply $\left\Vert y\right\Vert ^{2}=\left\Vert A_{n}y\right\Vert ^{2}+\left\Vert D_{A_{n}}y\right\Vert ^{2}$
with $y=T_{n-1}x$, and sum from $n=1$ to $N$. Polarization gives
the operator identity. 
\end{proof}

\begin{defn}[Effectiveness/admissibility]
 \label{def:2-2}Let $A_{1},A_{2},\dots$ be contractions on $H$
and let 
\[
T_{n}=A_{n}A_{n-1}\cdots A_{1}.
\]
We say that the sequence $\left\{ A_{n}\right\} $ is effective (or
admissible) if $\left\{ T_{n}\right\} $ converges in the strong operator
topology ($T_{n}\xrightarrow{s}T_{\infty}$), i.e., if there exists
$T_{\infty}\in B\left(H\right)$ such that $T_{n}x\to T_{\infty}x$
for every $x\in H$.
\end{defn}

\subsection{Kaczmarz-$\lambda P$ schemes}\label{subsec:2-2}

We now specialize the general contraction telescope of \prettyref{prop:2-1}
to the projection-driven setting that underlies Kaczmarz-type algorithms.
The basic building blocks are contractions of the form $I-\lambda P$
with $P$ an orthogonal projection and $0<\lambda<2$. Iterating such
factors produces products $\widetilde{T}_{n}$ whose action on vectors
encodes the algorithm, while the associated energy residuals $\widetilde{T}^{*}_{n}\widetilde{T}_{n}$
satisfy an automatic monotone decay governed by \prettyref{prop:2-1}.
The central issue is then whether this energy control upgrades to
convergence of the products themselves, i.e., effectiveness in the
sense of \prettyref{def:2-2}. The next example gives the basic defect
identity and the resulting telescoping formula in the $\lambda$-relaxed
projection case. 
\begin{example}[\cite{jeong2025}]
\label{exa:2-3} Let $P=P^{*}=P^{2}$ be an orthogonal projection
and fix $0<\lambda<2$. Define $A=I-\lambda P$. Then $A$ is a contraction
since 
\[
A^{*}A=A^{2}=\left(I-\lambda P\right)^{2}=I-\lambda\left(2-\lambda\right)P\le I.
\]
Consequently, 
\[
D^{2}_{A}=I-A^{*}A=\lambda\left(2-\lambda\right)P,\qquad D_{A}=\sqrt{\lambda\left(2-\lambda\right)}P,
\]
and the defect identity becomes 
\[
\left\Vert x\right\Vert ^{2}=\left\Vert \left(I-\lambda P\right)x\right\Vert ^{2}+\lambda\left(2-\lambda\right)\left\Vert Px\right\Vert ^{2}.
\]

Now take a sequence of orthogonal projections $P_{1},P_{2},\dots$
and form the $\lambda$-relaxed product 
\[
\widetilde{T}_{0}=I,\qquad\widetilde{T}_{n}=\left(I-\lambda P_{n}\right)\cdots\left(I-\lambda P_{1}\right).
\]
With the corresponding normalization 
\[
\widetilde{Q}_{n}:=\lambda P_{n}\widetilde{T}_{n-1},
\]
\prettyref{prop:2-1} yields the telescoping identity 
\[
I-\widetilde{T}^{*}_{N}\widetilde{T}_{N}=\lambda\left(2-\lambda\right)\sum^{N}_{n=1}\widetilde{T}^{*}_{n-1}P_{n}\widetilde{T}_{n-1}=\frac{2-\lambda}{\lambda}\sum^{N}_{n=1}\widetilde{Q}^{*}_{n}\widetilde{Q}_{n}.
\]

In this contraction-driven setting, a central convergence question
is whether the $\lambda$-relaxed products $\widetilde{T}_{n}$ are
effective in the sense of \prettyref{def:2-2}.  By contrast, the
associated energy residuals $\widetilde{T}^{*}_{n}\widetilde{T}_{n}$
always converge strongly: since each factor $I-\lambda P_{n}$ is
a contraction, one has 
\[
0\le\widetilde{T}^{*}_{n+1}\widetilde{T}_{n+1}=\widetilde{T}^{*}_{n}\left(I-\lambda P_{n+1}\right)^{*}\left(I-\lambda P_{n+1}\right)\widetilde{T}_{n}\le\widetilde{T}^{*}_{n}\widetilde{T}_{n},
\]
hence $\widetilde{T}^{*}_{n}\widetilde{T}_{n}\downarrow S_{\infty}$
for some $S_{\infty}\in B(H)_{+}$. The telescoping identity above
refines this monotonicity by identifying the dissipation at each step:
\[
\left\Vert x\right\Vert ^{2}-\Vert\widetilde{T}_{N}x\Vert^{2}=\frac{2-\lambda}{\lambda}\sum^{N}_{n=1}\Vert\widetilde{Q}_{n}x\Vert^{2}.
\]
The admissibility/effectiveness hypotheses in the cited works are
additional structural conditions on the projection sequence (and on
$\lambda$) that upgrade this energy control to strong convergence
statements for the products $\widetilde{T}_{n}$ themselves; see \cite{MR1881441,MR4126821,jeong2025}. 
\end{example}

Allowing a relaxation schedule $\{\lambda_{n}\}$ leads to a variable-step
$\lambda_{n}P_{n}$ scheme; the next theorem gives the corresponding
telescoping identity and a summability criterion guaranteeing effectiveness.

\subsection{Telescoping, summability criteria, and effectiveness}
\begin{thm}
\label{thm:2-4} Let $P_{1},P_{2},\dots$ be orthogonal projections
on $H$, and let $0<\lambda_{n}<2$ be a sequence. Set 
\[
A_{n}:=I-\lambda_{n}P_{n},\qquad T_{0}=I,\qquad T_{n}:=A_{n}A_{n-1}\cdots A_{1}.
\]
Then each $A_{n}$ is a contraction and one has the telescoping identity
\begin{equation}
I-T^{*}_{N}T_{N}=\sum^{N}_{n=1}\lambda_{n}\left(2-\lambda_{n}\right)\,T^{*}_{n-1}P_{n}T_{n-1}=\sum^{N}_{n=1}Q^{*}_{n}Q_{n},\label{eq:2-1}
\end{equation}
where 
\[
Q_{n}:=\sqrt{\lambda_{n}\left(2-\lambda_{n}\right)}\,P_{n}T_{n-1}.
\]
In particular, for every $x\in H$ and every $N\ge1$, 
\begin{equation}
\left\Vert x\right\Vert ^{2}-\left\Vert T_{N}x\right\Vert ^{2}=\sum^{N}_{n=1}\lambda_{n}\left(2-\lambda_{n}\right)\left\Vert P_{n}T_{n-1}x\right\Vert ^{2}=\sum^{N}_{n=1}\left\Vert Q_{n}x\right\Vert ^{2}.\label{eq:2-2}
\end{equation}

Assume in addition that 
\begin{equation}
\sum^{\infty}_{n=1}\frac{\lambda_{n}}{2-\lambda_{n}}<\infty.\label{eq:2-3}
\end{equation}
Then the products $\{T_{n}\}$ are effective in the sense of \prettyref{def:2-2},
i.e., $T_{n}\xrightarrow{s}T_{\infty}$ for some $T_{\infty}\in B(H)$.
\end{thm}

\begin{proof}
Since $P^{2}_{n}=P_{n}$ and $P^{*}_{n}=P_{n}$, one has 
\[
A^{*}_{n}A_{n}=A^{2}_{n}=\left(I-\lambda_{n}P_{n}\right)^{2}=I-\lambda_{n}\left(2-\lambda_{n}\right)P_{n}\le I,
\]
so $A_{n}$ is a contraction and $D^{2}_{A_{n}}=I-A^{*}_{n}A_{n}=\lambda_{n}(2-\lambda_{n})P_{n}$.
Applying \prettyref{prop:2-1} with these $A_{n}$ gives \eqref{eq:2-1}
and hence \eqref{eq:2-2}.

For effectiveness, fix $x\in H$ and write the step difference 
\[
T_{n}x-T_{n-1}x=\left(A_{n}-I\right)T_{n-1}x=-\lambda_{n}P_{n}T_{n-1}x.
\]
By Cauchy-Schwarz, 
\begin{align*}
\sum^{\infty}_{n=1}\left\Vert T_{n}x-T_{n-1}x\right\Vert  & =\sum^{\infty}_{n=1}\lambda_{n}\left\Vert P_{n}T_{n-1}x\right\Vert \\
 & =\sum^{\infty}_{n=1}\sqrt{\frac{\lambda_{n}}{2-\lambda_{n}}}\;\sqrt{\lambda_{n}\left(2-\lambda_{n}\right)}\left\Vert P_{n}T_{n-1}x\right\Vert \\
 & \le\left(\sum^{\infty}_{n=1}\frac{\lambda_{n}}{2-\lambda_{n}}\right)^{1/2}\left(\sum^{\infty}_{n=1}\lambda_{n}\left(2-\lambda_{n}\right)\left\Vert P_{n}T_{n-1}x\right\Vert ^{2}\right)^{1/2}.
\end{align*}
The second series is bounded by \eqref{eq:2-2}: 
\[
\sum^{\infty}_{n=1}\lambda_{n}\left(2-\lambda_{n}\right)\left\Vert P_{n}T_{n-1}x\right\Vert ^{2}\le\left\Vert x\right\Vert ^{2}.
\]
Thus \eqref{eq:2-3} implies $\sum_{n}\left\Vert T_{n}x-T_{n-1}x\right\Vert <\infty$,
so $\{T_{n}x\}$ is Cauchy in $H$. Since $x$ was arbitrary, $T_{n}\xrightarrow{s}T_{\infty}$. 
\end{proof}

The summability condition \eqref{eq:2-3} is a purely one-dimensional
constraint on the relaxation schedule $\{\lambda_{n}\}$; it gives
a projection-independent sufficient condition for effectiveness, while
finer criteria necessarily depend on the geometry of the projections. 
\begin{rem*}[Why variable relaxation can matter]
 In learning-style uses of Kaczmarz updates, the relaxation parameter
acts as a step size. In particular, the analysis in \cite{jeong2025}
exhibits an explicit $\lambda$-dependence in regret bounds, separating
a ``learning term'' from a ``noise floor,'' and motivates noise-aware
tuning of $\lambda$ across a run. Allowing $\lambda_{n}$ to vary
therefore accommodates standard step-size scheduling (e.g., aggressive
early updates followed by damping), while \eqref{eq:2-3} provides
a simple projection-independent sufficient condition ensuring effectiveness
of such schedules. 
\end{rem*}

\subsection{Multichannel tree splitting}\label{subsec:2-4}

Let $A_{1},\dots,A_{d}\in B\left(H\right)$ be bounded operators such
that 
\[
\sum^{d}_{i=1}A^{*}_{i}A_{i}\le I.
\]
Equivalently, $\left(A_{1},\dots,A_{d}\right)$ is a column contraction.
Set 
\[
D:=\left(I-\sum^{d}_{i=1}A^{*}_{i}A_{i}\right)^{1/2}.
\]
The basic observation is that this gives a multichannel splitting
of the norm.
\begin{lem}
\label{lem:2-1}For every $x\in H$, 
\[
\left\Vert x\right\Vert ^{2}=\sum^{d}_{i=1}\left\Vert A_{i}x\right\Vert ^{2}+\left\Vert Dx\right\Vert ^{2}.
\]
\end{lem}

\begin{proof}
We have 
\[
\sum^{d}_{i=1}\left\Vert A_{i}x\right\Vert ^{2}+\left\Vert Dx\right\Vert ^{2}=\left\langle x,\left(\sum\nolimits^{d}_{i=1}A^{*}_{i}A_{i}+D^{2}\right)x\right\rangle =\left\langle x,Ix\right\rangle =\left\Vert x\right\Vert ^{2}.
\]
\end{proof}

We now iterate this identity along a rooted $d$-ary tree.
\begin{defn}
\label{def:2-6}Let $\mathcal{T}$ be the rooted $d$-ary tree. We
label each node by a word 
\[
\alpha=i_{1}\cdots i_{n},\qquad i_{k}\in\left\{ 1,\dots,d\right\} ,
\]
with root $\emptyset$ (the empty word) and depth $\left|\alpha\right|=n$.
For a word $\alpha=i_{1}\cdots i_{n}$ we define 
\[
A_{\alpha}:=A_{i_{n}}\cdots A_{i_{1}},\qquad A_{\emptyset}:=I.
\]
For $x\in H$ and a node $\alpha$, we set 
\[
E_{\alpha}(x):=\left\Vert A_{\alpha}x\right\Vert ^{2},\qquad\Delta_{\alpha}(x):=\left\Vert DA_{\alpha}x\right\Vert ^{2}.
\]
For each level $N\ge0$ we define the level energy 
\[
L_{N}(x):=\sum_{\left|\alpha\right|=N}E_{\alpha}(x)=\sum_{\left|\alpha\right|=N}\left\Vert A_{\alpha}x\right\Vert ^{2}.
\]
\end{defn}

With this notation, the multichannel splitting gives a refinement
identity along the tree.
\begin{prop}
\label{prop:2-3}For every $x\in H$ and every integer $N\ge1$, 
\begin{equation}
\left\Vert x\right\Vert ^{2}=\sum_{\left|\alpha\right|=N}\left\Vert A_{\alpha}x\right\Vert ^{2}+\sum^{N-1}_{n=0}\sum_{\left|\alpha\right|=n}\left\Vert DA_{\alpha}x\right\Vert ^{2}.\label{eq:2-4}
\end{equation}
\end{prop}

\begin{proof}
Fix a node $\alpha$. Apply \prettyref{lem:2-1} to the vector $A_{\alpha}x$:
\[
\left\Vert A_{\alpha}x\right\Vert ^{2}=\sum^{d}_{i=1}\left\Vert A_{i}A_{\alpha}x\right\Vert ^{2}+\left\Vert DA_{\alpha}x\right\Vert ^{2}=\sum^{d}_{i=1}\left\Vert A_{\alpha i}x\right\Vert ^{2}+\Delta_{\alpha}(x),
\]
where $\alpha i$ is the child of $\alpha$ obtained by appending
$i$. Now sum over all nodes at depth $n$. The left-hand side becomes
$L_{n}(x)$, and the first term on the right becomes $L_{n+1}(x)$.
Thus 
\[
L_{n}(x)=L_{n+1}(x)+\sum_{\left|\alpha\right|=n}\Delta_{\alpha}(x)
\]
for every $n\ge0$. Summing this identity from $n=0$ to $n=N-1$
gives 
\[
L_{0}(x)=L_{N}(x)+\sum^{N-1}_{n=0}\sum_{\left|\alpha\right|=n}\Delta_{\alpha}(x).
\]
Since $L_{0}(x)=\left\Vert x\right\Vert ^{2}$, the proposition follows. 
\end{proof}

A few immediate consequences will be useful later.
\begin{cor}[monotonicity and limit]
\label{cor:2-4} For every $x\in H$ the sequence $\left(L_{N}(x)\right)_{N\ge0}$
is decreasing, and the limit 
\[
L_{\infty}(x):=\lim_{N\to\infty}L_{N}(x)
\]
exists in $\left[0,\left\Vert x\right\Vert ^{2}\right]$. Moreover,
\[
\sum_{n\ge0}\sum_{\left|\alpha\right|=n}\left\Vert DA_{\alpha}x\right\Vert ^{2}=\left\Vert x\right\Vert ^{2}-L_{\infty}(x).
\]
\end{cor}

\begin{proof}
From the proof of \prettyref{prop:2-3} we have 
\[
L_{N+1}(x)=L_{N}(x)-\sum_{\left|\alpha\right|=N}\Delta_{\alpha}(x)\le L_{N}(x),
\]
so $\left(L_{N}(x)\right)$ is decreasing and bounded below by $0$,
hence convergent. Taking $N\to\infty$ in \eqref{eq:2-4} gives the
last statement. 
\end{proof}

\begin{cor}[column isometric case]
\label{cor:2-5} If $\sum^{d}_{i=1}A^{*}_{i}A_{i}=I$, then $D=0$
and for every $x\in H$ and every $N\ge0$, 
\[
\left\Vert x\right\Vert ^{2}=\sum_{\left|\alpha\right|=N}\left\Vert A_{\alpha}x\right\Vert ^{2}.
\]
In particular $L_{N}(x)=\left\Vert x\right\Vert ^{2}$ for all $N$,
and $L_{\infty}(x)=\left\Vert x\right\Vert ^{2}$. 
\end{cor}

\begin{proof}
If $\sum A^{*}_{i}A_{i}=I$ then $D=0$, so all $\Delta_{\alpha}(x)$
vanish and \eqref{eq:2-4} reduces to the stated identity. 
\end{proof}

\begin{cor}
\label{cor:2-6}Assume there is a constant $c\in(0,1)$ such that
\[
\sum^{d}_{i=1}A^{*}_{i}A_{i}\le cI.
\]
Then for every $x\in H$ and every $N\ge0$, 
\[
\sum_{\left|\alpha\right|=N}\left\Vert A_{\alpha}x\right\Vert ^{2}\le c^{N}\left\Vert x\right\Vert ^{2}.
\]
In particular $L_{\infty}(x)=0$, and 
\[
\sum_{n\ge0}\ \sum_{\left|\alpha\right|=n}\left\Vert DA_{\alpha}x\right\Vert ^{2}=\left\Vert x\right\Vert ^{2}.
\]
\end{cor}

\begin{proof}
The hypothesis implies $\sum^{d}_{i=1}\left\Vert A_{i}y\right\Vert ^{2}\le c\left\Vert y\right\Vert ^{2}$
for every $y\in H$. For fixed $N$ apply this with $y=A_{\alpha}x$
and sum over all words $\alpha$ of length $N$. Then 
\[
L_{N+1}(x)=\sum_{\left|\beta\right|=N+1}\left\Vert A_{\beta}x\right\Vert ^{2}=\sum_{\left|\alpha\right|=N}\ \sum^{d}_{i=1}\left\Vert A_{i}A_{\alpha}x\right\Vert ^{2}\le c\sum_{\left|\alpha\right|=N}\left\Vert A_{\alpha}x\right\Vert ^{2}=cL_{N}(x).
\]
By induction we get $L_{N}(x)\le c^{N}L_{0}(x)=c^{N}\left\Vert x\right\Vert ^{2}$.
The remaining statements follow from \prettyref{cor:2-4}. 
\end{proof}

\section{Kernel interpolation}\label{sec:3}

The setting for our next section is that of kernels and their reproducing
kernel Hilbert spaces(RKHSs). For the benefit of readers, we offer
the following citations \cite{MR4905634,MR4302453,MR3251728,MR4509085,MR3526117,MR4619928,MR4582508,MR3860446}.

We now place the $\lambda_{n}P_{n}$ products from \prettyref{thm:2-4}
in a reproducing kernel Hilbert space (RKHS) setting, where the projections
arise from point evaluations. Let $X$ be a set and let $H_{k}$ be
an RKHS of complex-valued functions on $X$ with reproducing kernel
$k$. For $x\in X$ we write $k_{x}:=k\left(\cdot,x\right)\in H_{k}$,
so that 
\begin{equation}
f\left(x\right)=\left\langle k_{x},f\right\rangle ,\qquad f\in H_{k}.\label{eq:3-1}
\end{equation}
Fix a sequence $\left(x_{n}\right)_{n\ge1}\subset X$ such that $k\left(x_{n},x_{n}\right)>0$
for all $n$, and set 
\[
g_{n}:=\frac{k_{x_{n}}}{\left\Vert k_{x_{n}}\right\Vert }=\frac{k_{x_{n}}}{\sqrt{k\left(x_{n},x_{n}\right)}}\in H_{k}.
\]
Let $P_{n}=\left|g_{n}\right\rangle \left\langle g_{n}\right|$ denote
the orthogonal projection onto $\mathbb{C}g_{n}$, i.e. 
\[
P_{n}f=\left\langle g_{n},f\right\rangle g_{n},\qquad f\in H_{k}.
\]
Evaluating at $x_{n}$ and using the reproducing property \eqref{eq:3-1}
shows that this projection is simply 
\[
P_{n}f=\frac{f\left(x_{n}\right)}{k\left(x_{n},x_{n}\right)}k_{x_{n}}.
\]

Now fix $0<\lambda_{n}<2$, set $A_{n}:=I-\lambda_{n}P_{n}$, and
let $T_{n}=A_{n}A_{n-1}\cdots A_{1}$ as in \prettyref{thm:2-4}.
Since the $P_{n}$ are orthogonal projections, each $A_{n}$ is a
contraction, and \eqref{eq:2-2} becomes an explicit evaluation formula
in $H_{k}$: 
\begin{align}
\left\Vert f\right\Vert ^{2}-\left\Vert T_{N}f\right\Vert ^{2} & =\sum^{N}_{n=1}\lambda_{n}\left(2-\lambda_{n}\right)\left\Vert P_{n}T_{n-1}f\right\Vert ^{2}\nonumber \\
 & =\sum^{N}_{n=1}\lambda_{n}\left(2-\lambda_{n}\right)\frac{\left|\left(T_{n-1}f\right)\left(x_{n}\right)\right|^{2}}{k\left(x_{n},x_{n}\right)}.\label{eq:3-2}
\end{align}
In particular, as $N\to\infty$ the partial sums on the right are
monotone increasing and remain bounded by $\left\Vert f\right\Vert ^{2}$,
hence 
\[
\sum^{\infty}_{n=1}\lambda_{n}\left(2-\lambda_{n}\right)\frac{\left|\left(T_{n-1}f\right)\left(x_{n}\right)\right|^{2}}{k\left(x_{n},x_{n}\right)}\le\left\Vert f\right\Vert ^{2},\qquad f\in H_{k}.
\]

To connect this with interpolation, we pass to the usual affine update.
Let $\left(y_{n}\right)_{n\ge1}\subset\mathbb{C}$ be prescribed data
and assume there exists $f^{\ast}\in H_{k}$ such that 
\[
f^{\ast}\left(x_{n}\right)=y_{n},\qquad n\ge1.
\]
Starting from $f_{0}=0$, define $\left(f_{n}\right)_{n\ge0}\subset H_{k}$
recursively by 
\begin{equation}
f_{n}=f_{n-1}+\lambda_{n}\left(y_{n}-f_{n-1}\left(x_{n}\right)\right)\frac{k_{x_{n}}}{k\left(x_{n},x_{n}\right)},\qquad n\ge1.\label{eq:3-3}
\end{equation}
Set $e_{n}:=f^{\ast}-f_{n}$. A short computation shows that the error
evolves by the same $\lambda_{n}P_{n}$ factors: 
\[
e_{n}=e_{n-1}-\lambda_{n}P_{n}e_{n-1}=\left(I-\lambda_{n}P_{n}\right)e_{n-1}=A_{n}e_{n-1},
\]
so $e_{n}=T_{n}e_{0}=T_{n}f^{\ast}$. Substituting $f=f^{\ast}$ in
\eqref{eq:3-2} yields an identity in which the dissipated pieces
are exactly the interpolation residuals.
\begin{thm}
\label{thm:3-1} Assume the interpolation constraints $f^{\ast}\left(x_{n}\right)=y_{n}$
admit a solution $f^{\ast}\in H_{k}$, and let $\left(f_{n}\right)$
be defined by \eqref{eq:3-3}. Then for every $N\ge1$, 
\begin{equation}
\left\Vert f^{\ast}-f_{N}\right\Vert ^{2}=\left\Vert f^{\ast}\right\Vert ^{2}-\sum^{N}_{n=1}\lambda_{n}\left(2-\lambda_{n}\right)\frac{\left|y_{n}-f_{n-1}\left(x_{n}\right)\right|^{2}}{k\left(x_{n},x_{n}\right)}.\label{eq:3-4}
\end{equation}
In particular, 
\[
\sum^{\infty}_{n=1}\lambda_{n}\left(2-\lambda_{n}\right)\frac{\left|y_{n}-f_{n-1}\left(x_{n}\right)\right|^{2}}{k\left(x_{n},x_{n}\right)}\le\left\Vert f^{\ast}\right\Vert ^{2}.
\]
\end{thm}

\begin{proof}
We have $e_{n}=f^{\ast}-f_{n}=T_{n}f^{\ast}$. Moreover, 
\[
\left(T_{n-1}f^{\ast}\right)\left(x_{n}\right)=\left\langle k_{x_{n}},T_{n-1}f^{\ast}\right\rangle =\left\langle k_{x_{n}},e_{n-1}\right\rangle =y_{n}-f_{n-1}\left(x_{n}\right),
\]
since $e_{n-1}=f^{\ast}-f_{n-1}$ and $f^{\ast}\left(x_{n}\right)=y_{n}$.
Substituting $f=f^{\ast}$ in \eqref{eq:3-2} gives \eqref{eq:3-4}. 
\end{proof}

\begin{rem}
The identity \eqref{eq:3-4} gives an exact energy balance for the
$\lambda_{n}P_{n}$ dynamics: the error norm drops by summing the
squared residuals at the sampled points, with the weights $\lambda_{n}\left(2-\lambda_{n}\right)/k\left(x_{n},x_{n}\right)$.
For convergence one needs, in addition, an effectiveness hypothesis
in the sense of \prettyref{def:2-2}. If $T_{n}\xrightarrow{s}T_{\infty}$,
then $e_{n}=T_{n}f^{\ast}\to T_{\infty}f^{\ast}$ in $H_{k}$, and
hence $f_{n}\to f^{\ast}-T_{\infty}f^{\ast}$. In particular, if $T_{\infty}f^{\ast}=0$
(for example, if $T_{n}\xrightarrow{s}0$), then $f_{n}$ converges
in $H_{k}$ to the interpolant $f^{\ast}$. By \prettyref{thm:2-4},
the summability condition \eqref{eq:2-3} is a projection independent
sufficient hypothesis ensuring effectiveness of $\left\{ T_{n}\right\} $,
while sharper criteria in the present kernel setting depend on the
geometry of the kernel sections $\left\{ k_{x_{n}}\right\} $. 
\end{rem}

We next consider the effect of additive noise. The geometry of the
update is unchanged, but the exact identity \eqref{eq:3-4} is replaced
by a decomposition of the expected error into a dissipative part and
a noise term.

Assume now that the data have the form 
\[
y_{n}=f^{\ast}\left(x_{n}\right)+\varepsilon_{n},
\]
where $\varepsilon_{n}$ are complex-valued random variables. We impose
the usual square-integrable, mean-zero assumptions 
\begin{equation}
\mathbb{E}\left(\varepsilon_{n}\mid x_{n}\right)=0,\qquad\mathbb{E}\left(\left|\varepsilon_{n}\right|^{2}\mid x_{n}\right)\le\sigma^{2}\label{eq:3-5}
\end{equation}
for some $\sigma\ge0$, and we assume that the noises $\left\{ \varepsilon_{n}\right\} $
are independent across $n$. The relaxed update \eqref{eq:3-3} is
left unchanged, and we continue to write $e_{n}=f^{\ast}-f_{n}$.
\begin{prop}
\label{prop:3-2} Let $\left\{ f_{n}\right\} $ be defined by \eqref{eq:3-3}
with $y_{n}=f^{\ast}\left(x_{n}\right)+\varepsilon_{n}$, where $f^{\ast}\in H_{k}$
and the noises $\left\{ \varepsilon_{n}\right\} $ satisfy \eqref{eq:3-5}.
Set $e_{n}=f^{\ast}-f_{n}$. Then for every $N\ge1$, 
\begin{equation}
\mathbb{E}\left\Vert e_{N}\right\Vert ^{2}=\left\Vert f^{\ast}\right\Vert ^{2}-\sum^{N}_{n=1}\lambda_{n}\left(2-\lambda_{n}\right)\mathbb{E}\left\Vert P_{n}e_{n-1}\right\Vert ^{2}+\sum^{N}_{n=1}\lambda^{2}_{n}\mathbb{E}\left(\frac{\left|\varepsilon_{n}\right|^{2}}{k\left(x_{n},x_{n}\right)}\right).\label{eq:3-6}
\end{equation}
\end{prop}

\begin{proof}
At step $n$ one has 
\[
y_{n}-f_{n-1}\left(x_{n}\right)=\left(f^{\ast}\left(x_{n}\right)-f_{n-1}\left(x_{n}\right)\right)+\varepsilon_{n}=e_{n-1}\left(x_{n}\right)+\varepsilon_{n},
\]
so \eqref{eq:3-3} gives 
\begin{align*}
e_{n} & =e_{n-1}-\lambda_{n}\left(e_{n-1}\left(x_{n}\right)+\varepsilon_{n}\right)\frac{k_{x_{n}}}{k\left(x_{n},x_{n}\right)}\\
 & =\left(I-\lambda_{n}P_{n}\right)e_{n-1}-\lambda_{n}\varepsilon_{n}\frac{k_{x_{n}}}{k\left(x_{n},x_{n}\right)}.
\end{align*}
Introduce 
\[
u_{n}:=\left(I-\lambda_{n}P_{n}\right)e_{n-1},\qquad v_{n}:=\lambda_{n}\varepsilon_{n}\frac{k_{x_{n}}}{k\left(x_{n},x_{n}\right)},
\]
so $e_{n}=u_{n}-v_{n}$. As in the deterministic case, 
\[
\left\Vert u_{n}\right\Vert ^{2}=\left\Vert e_{n-1}\right\Vert ^{2}-\lambda_{n}\left(2-\lambda_{n}\right)\left\Vert P_{n}e_{n-1}\right\Vert ^{2},
\]
and 
\[
\left\Vert v_{n}\right\Vert ^{2}=\lambda^{2}_{n}\left|\varepsilon_{n}\right|^{2}\frac{\left\Vert k_{x_{n}}\right\Vert ^{2}}{k\left(x_{n},x_{n}\right)^{2}}=\lambda^{2}_{n}\frac{\left|\varepsilon_{n}\right|^{2}}{k\left(x_{n},x_{n}\right)}.
\]
We expand 
\[
\left\Vert e_{n}\right\Vert ^{2}=\left\Vert u_{n}-v_{n}\right\Vert ^{2}=\left\Vert u_{n}\right\Vert ^{2}+\left\Vert v_{n}\right\Vert ^{2}-\left\langle u_{n},v_{n}\right\rangle -\left\langle v_{n},u_{n}\right\rangle .
\]
Conditioning on $e_{n-1}$ and $x_{n}$, the vector $u_{n}$ is fixed,
while $v_{n}$ is linear in $\varepsilon_{n}$. By \eqref{eq:3-5},
\[
\mathbb{E}\left(\varepsilon_{n}\mid e_{n-1},x_{n}\right)=\mathbb{E}\left(\varepsilon_{n}\mid x_{n}\right)=0,
\]
and similarly $\mathbb{E}\left(\overline{\varepsilon_{n}}\mid e_{n-1},x_{n}\right)=0$,
so both mixed terms have zero conditional mean. Taking conditional
expectation therefore yields 
\begin{align*}
\mathbb{E}\left(\left\Vert e_{n}\right\Vert ^{2}\mid e_{n-1},x_{n}\right) & =\left\Vert e_{n-1}\right\Vert ^{2}-\lambda_{n}\left(2-\lambda_{n}\right)\left\Vert P_{n}e_{n-1}\right\Vert ^{2}\\
 & \qquad+\lambda^{2}_{n}\mathbb{E}\left(\frac{\left|\varepsilon_{n}\right|^{2}}{k\left(x_{n},x_{n}\right)}\mid x_{n}\right).
\end{align*}
Summing these relations over $n=1,\dots,N$ and taking full expectation
gives 
\[
\mathbb{E}\left\Vert e_{N}\right\Vert ^{2}=\mathbb{E}\left\Vert e_{0}\right\Vert ^{2}-\sum^{N}_{n=1}\lambda_{n}\left(2-\lambda_{n}\right)\mathbb{E}\left\Vert P_{n}e_{n-1}\right\Vert ^{2}+\sum^{N}_{n=1}\lambda^{2}_{n}\mathbb{E}\left(\frac{\left|\varepsilon_{n}\right|^{2}}{k\left(x_{n},x_{n}\right)}\right).
\]
Since $e_{0}=f^{\ast}$, this is exactly \eqref{eq:3-6}. 
\end{proof}

The first sum in \eqref{eq:3-6} is the same dissipative term as in
\eqref{eq:3-4}, now averaged over the randomness in the data, while
the second sum measures the accumulated noise. A simple uniform lower
bound on the kernel diagonals turns \eqref{eq:3-6} into a bound on
the expected error.
\begin{cor}
\label{cor:3-3} In addition to the assumptions of \prettyref{prop:3-2},
suppose there is a constant $\kappa>0$ such that 
\begin{equation}
k\left(x_{n},x_{n}\right)\ge\kappa,\qquad n\ge1.\label{eq:3-8}
\end{equation}
Then for every $N\ge1$, 
\begin{equation}
\mathbb{E}\left\Vert e_{N}\right\Vert ^{2}\le\left\Vert f^{\ast}\right\Vert ^{2}+\frac{\sigma^{2}}{\kappa}\sum^{N}_{n=1}\lambda^{2}_{n}.\label{eq:3-9}
\end{equation}
In particular, if $\sum^{\infty}_{n=1}\lambda^{2}_{n}<\infty$, then
$\sup_{N\ge1}\mathbb{E}\left\Vert e_{N}\right\Vert ^{2}<\infty$. 
\end{cor}

\begin{proof}
Starting from \eqref{eq:3-6} we drop the dissipative term to obtain
\[
\mathbb{E}\left\Vert e_{N}\right\Vert ^{2}\le\left\Vert f^{\ast}\right\Vert ^{2}+\sum^{N}_{n=1}\lambda^{2}_{n}\mathbb{E}\left(\frac{\left|\varepsilon_{n}\right|^{2}}{k\left(x_{n},x_{n}\right)}\right).
\]
By \eqref{eq:3-5} and \eqref{eq:3-8}, 
\[
\mathbb{E}\left(\frac{\left|\varepsilon_{n}\right|^{2}}{k\left(x_{n},x_{n}\right)}\right)\le\frac{1}{\kappa}\mathbb{E}\left|\varepsilon_{n}\right|^{2}\le\frac{\sigma^{2}}{\kappa},
\]
and so 
\[
\sum^{N}_{n=1}\lambda^{2}_{n}\mathbb{E}\left(\frac{\left|\varepsilon_{n}\right|^{2}}{k\left(x_{n},x_{n}\right)}\right)\le\frac{\sigma^{2}}{\kappa}\sum^{N}_{n=1}\lambda^{2}_{n}.
\]
This yields \eqref{eq:3-9}. The final statement follows immediately
from the assumption $\sum_{n}\lambda^{2}_{n}<\infty$. 
\end{proof}

\section{Random Dropouts (stochastic sampling)}\label{sec:4}

In many applications the update is not executed at every step: samples
may be missing, sensors may fail to transmit, or one may deliberately
skip updates as a form of regularization. A simple way to model this
is to introduce a Bernoulli variable which randomly turns the update
on and off. The point here is that the basic energy identity survives
pathwise, and this leads to clean supermartingale and stability statements
with only a little extra work.

Let $H$ be a Hilbert space and let $\left\{ P_{n}\right\} _{n\ge1}$
be a sequence of orthogonal projections on $H$. Fix $\lambda\in\left(0,2\right)$
and let $\left\{ \eta_{n}\right\} _{n\ge1}$ be $\left\{ 0,1\right\} $-valued
random variables. We consider the intermittent relaxation 
\[
x_{n}=x_{n-1}-\lambda\eta_{n}P_{n}x_{n-1},\qquad n\ge1.
\]
Equivalently, with $A_{n}=I-\lambda\eta_{n}P_{n}$ we have $x_{n}=A_{n}x_{n-1}$
for all $n\ge1$.
\begin{cor}
For every realization of $\left\{ \eta_{n}\right\} $ and every $x\in H$,
\[
\left\Vert A_{n}x\right\Vert ^{2}=\left\Vert x\right\Vert ^{2}-\eta_{n}\lambda\left(2-\lambda\right)\left\Vert P_{n}x\right\Vert ^{2}.
\]
In particular $\left\Vert x_{n}\right\Vert $ is nonincreasing, and
\[
\left\Vert x_{0}\right\Vert ^{2}-\left\Vert x_{N}\right\Vert ^{2}=\sum^{N}_{n=1}\eta_{n}\lambda\left(2-\lambda\right)\left\Vert P_{n}x_{n-1}\right\Vert ^{2}
\]
for every $N\ge1$. 
\end{cor}

\begin{proof}
This is immediate from \prettyref{exa:2-3}, since $\eta_{n}\in\left\{ 0,1\right\} $
and 
\[
\left(I-\lambda\eta_{n}P_{n}\right)^{2}=I-\eta_{n}\lambda\left(2-\lambda\right)P_{n}.
\]
\end{proof}

These identities are deterministic in the sense that they hold pathwise.
To obtain convergence information it is convenient to place the iteration
in a filtration. Let $\left\{ \mathcal{F}_{n}\right\} _{n\ge0}$ be
a filtration such that $x_{n-1}$ is $\mathcal{F}_{n-1}$-measurable
and $\left(P_{n},\eta_{n}\right)$ is $\mathcal{F}_{n}$-measurable.
Set 
\[
p_{n}=\mathbb{E}\left[\eta_{n}\mid\mathcal{F}_{n-1}\right].
\]

\begin{prop}
With notation as above, 
\[
\mathbb{E}\left[\left\Vert x_{n}\right\Vert ^{2}\mid\mathcal{F}_{n-1}\right]=\left\Vert x_{n-1}\right\Vert ^{2}-\lambda\left(2-\lambda\right)\mathbb{E}\left[\eta_{n}\left\Vert P_{n}x_{n-1}\right\Vert ^{2}\mid\mathcal{F}_{n-1}\right].
\]
In particular $\left\{ \left\Vert x_{n}\right\Vert ^{2}\right\} $
is a nonnegative supermartingale, and 
\[
\sum_{n\ge1}\eta_{n}\lambda\left(2-\lambda\right)\left\Vert P_{n}x_{n-1}\right\Vert ^{2}<\infty
\]
almost surely. 
\end{prop}

\begin{proof}
Take conditional expectation of the pathwise identity in the first
corollary, using that $x_{n-1}$ is $\mathcal{F}_{n-1}$-measurable.
The supermartingale property is immediate. The almost sure summability
follows from the telescoping identity and the boundedness of $\left\Vert x_{n}\right\Vert $. 
\end{proof}

To identify the limit point we impose a geometric hypothesis on the
sequence of projections. We fix a deterministic common fixed-point
subspace:
\begin{assumption}
\label{ass:M} There exists a fixed closed subspace $M\subseteq H$
such that 
\begin{equation}
P_{n}m=0\qquad\text{for all }m\in M
\end{equation}
almost surely, for every $n\ge1$.
\end{assumption}

We write $P_{M}$ for the orthogonal projection onto $M$ and $M^{\perp}$
for its orthogonal complement. 
\begin{thm}
\label{thm:4-4} Assume \prettyref{ass:M}. Suppose there exists $\beta>0$
such that for every $x\in M^{\perp}$, 
\[
\left\langle \mathbb{E}\left[\eta_{n}P_{n}\mid\mathcal{F}_{n-1}\right]x,x\right\rangle \ge\beta\left\Vert x\right\Vert ^{2}
\]
almost surely for all $n\ge1$. Then for every initial $x_{0}\in H$,
\[
\mathbb{E}\left[\left\Vert x_{n}-P_{M}x_{0}\right\Vert ^{2}\right]\le\left(1-\beta\lambda\left(2-\lambda\right)\right)^{n}\left\Vert x_{0}-P_{M}x_{0}\right\Vert ^{2},
\]
so $x_{n}\to P_{M}x_{0}$ in $L^{2}$. Moreover, 
\[
x_{n}\to P_{M}x_{0}\qquad\text{almost surely.}
\]
\end{thm}

\begin{proof}
Write $e_{n}=x_{n}-P_{M}x_{0}$. Since $P_{M}x_{0}\in M$ and $P_{n}m=0$
for all $m\in M$, we have $P_{n}e_{n-1}=P_{n}x_{n-1}$. Moreover,
$e_{n}\in M^{\perp}$ for all $n$, since $x_{n}-P_{M}x_{0}$ differs
from $x_{n}$ by an element of $M$ and each update subtracts a vector
in $\text{ran}P_{n}\subseteq M^{\perp}$. Indeed, if $y=P_{n}z$ and
$m\in M$ then 
\[
\left\langle y,m\right\rangle =\left\langle P_{n}z,m\right\rangle =\left\langle z,P_{n}m\right\rangle =0,
\]
so $\text{ran}P_{n}\subseteq M^{\perp}$.

The conditional expectation identity gives 
\[
\mathbb{E}\left[\left\Vert e_{n}\right\Vert ^{2}\mid\mathcal{F}_{n-1}\right]=\left\Vert e_{n-1}\right\Vert ^{2}-\lambda\left(2-\lambda\right)\mathbb{E}\left[\eta_{n}\left\Vert P_{n}e_{n-1}\right\Vert ^{2}\mid\mathcal{F}_{n-1}\right].
\]
Since $\left\Vert P_{n}e_{n-1}\right\Vert ^{2}=\left\langle P_{n}e_{n-1},e_{n-1}\right\rangle $
and $e_{n-1}$ is $\mathcal{F}_{n-1}$-measurable, we have 
\[
\mathbb{E}\left[\eta_{n}\left\Vert P_{n}e_{n-1}\right\Vert ^{2}\mid\mathcal{F}_{n-1}\right]=\left\langle \mathbb{E}\left[\eta_{n}P_{n}\mid\mathcal{F}_{n-1}\right]e_{n-1},e_{n-1}\right\rangle .
\]
Using the coercivity hypothesis yields 
\[
\mathbb{E}\left[\left\Vert e_{n}\right\Vert ^{2}\mid\mathcal{F}_{n-1}\right]\le\left(1-\beta\lambda\left(2-\lambda\right)\right)\left\Vert e_{n-1}\right\Vert ^{2}.
\]
Taking expectations and iterating gives 
\[
\mathbb{E}\left\Vert e_{n}\right\Vert ^{2}\le\left(1-\beta\lambda\left(2-\lambda\right)\right)^{n}\left\Vert e_{0}\right\Vert ^{2},
\]
which is the stated $L^{2}$ bound.

For almost sure convergence, set $c=\beta\lambda\left(2-\lambda\right)\in\left(0,1\right)$.
The inequality above implies that $\left\{ (1-c)^{-n}\left\Vert e_{n}\right\Vert ^{2}\right\} $
is a nonnegative supermartingale, hence it converges almost surely
to a finite random limit. In particular $\left\Vert e_{n}\right\Vert ^{2}$
converges almost surely. On the other hand, 
\[
\mathbb{E}\left\Vert e_{n}\right\Vert ^{2}\le(1-c)^{n}\left\Vert e_{0}\right\Vert ^{2}
\]
so $\sum_{n\ge1}\mathbb{E}\left\Vert e_{n}\right\Vert ^{2}<\infty$.
By Markov's inequality and the Borel-Cantelli lemma, $\left\Vert e_{n}\right\Vert \to0$
almost surely. Thus $x_{n}\to P_{M}x_{0}$ almost surely. 
\end{proof}

\begin{rem}
A common situation is that $\eta_{n}$ and $P_{n}$ are conditionally
independent given $\mathcal{F}_{n-1}$, and $p_{n}=\mathbb{E}[\eta_{n}\mid\mathcal{F}_{n-1}]\ge p>0$
almost surely, while there exists $\alpha>0$ such that 
\[
\left\langle \mathbb{E}\left[P_{n}\mid\mathcal{F}_{n-1}\right]x,x\right\rangle \ge\alpha\left\Vert x\right\Vert ^{2},\qquad x\in M^{\perp}.
\]
Then 
\[
\mathbb{E}\left[\eta_{n}P_{n}\mid\mathcal{F}_{n-1}\right]=\mathbb{E}\left[\eta_{n}\mid\mathcal{F}_{n-1}\right]\mathbb{E}\left[P_{n}\mid\mathcal{F}_{n-1}\right]
\]
and the hypotheses above imply 
\[
\left\langle \mathbb{E}\left[\eta_{n}P_{n}\mid\mathcal{F}_{n-1}\right]x,x\right\rangle \ge p\alpha\left\Vert x\right\Vert ^{2},\qquad x\in M^{\perp}.
\]
Thus \prettyref{thm:4-4} applies with $\beta=p\alpha$. 
\end{rem}

\begin{rem}
Intermittent updates do not change the basic dissipation mechanism;
they scale the effective contraction rate through the conditional
average operator $\mathbb{E}\left[\eta_{n}P_{n}\mid\mathcal{F}_{n-1}\right]$
on $M^{\perp}$. In particular, \prettyref{thm:4-4} shows that a
uniform coercivity bound on these conditional averages yields geometric
$L^{2}$ convergence and almost sure convergence to $P_{M}x_{0}$.
A deterministic analogue, closer in spirit to cyclic Kaczmarz, would
be to assume a windowed lower frame bound for the family $\left\{ P_{n}\right\} $.
We leave that case aside here. 
\end{rem}

It is convenient to record a variable-step version which combines
dropout with a predictable relaxation schedule.
\begin{thm}
\label{thm:4-7} Let $\left\{ \lambda_{n}\right\} _{n\ge1}$ be a
predictable sequence, i.e.\ $\lambda_{n}$ is $\mathcal{F}_{n-1}$-measurable
for each $n$, and assume $0<\lambda_{n}<2$ almost surely. Consider
the intermittent update 
\[
x_{n}=x_{n-1}-\lambda_{n}\eta_{n}P_{n}x_{n-1},\qquad n\ge1.
\]
Then the pathwise identity holds: 
\[
\left\Vert x_{n}\right\Vert ^{2}=\left\Vert x_{n-1}\right\Vert ^{2}-\eta_{n}\lambda_{n}\left(2-\lambda_{n}\right)\left\Vert P_{n}x_{n-1}\right\Vert ^{2}.
\]
Assume moreover that there exists $\gamma>0$ such that for every
$x\in M^{\perp}$, 
\[
\left\langle \mathbb{E}\left[\eta_{n}\lambda_{n}\left(2-\lambda_{n}\right)P_{n}\mid\mathcal{F}_{n-1}\right]x,x\right\rangle \ge\gamma\left\Vert x\right\Vert ^{2}
\]
almost surely for all $n$. Then for every $x_{0}\in H$, 
\[
\mathbb{E}\left[\left\Vert x_{n}-P_{M}x_{0}\right\Vert ^{2}\right]\le\left(1-\gamma\right)^{n}\left\Vert x_{0}-P_{M}x_{0}\right\Vert ^{2},
\]
and $x_{n}\to P_{M}x_{0}$ almost surely.
\end{thm}

\begin{proof}
The pathwise identity is obtained by applying \prettyref{exa:2-3}
with step size $\lambda_{n}\eta_{n}$. For the contraction estimate,
let $e_{n}=x_{n}-P_{M}x_{0}$. As above, 
\[
\mathbb{E}\left[\left\Vert e_{n}\right\Vert ^{2}\mid\mathcal{F}_{n-1}\right]=\left\Vert e_{n-1}\right\Vert ^{2}-\mathbb{E}\left[\eta_{n}\lambda_{n}\left(2-\lambda_{n}\right)\left\Vert P_{n}e_{n-1}\right\Vert ^{2}\mid\mathcal{F}_{n-1}\right],
\]
and the last term equals 
\[
\left\langle \mathbb{E}\left[\eta_{n}\lambda_{n}\left(2-\lambda_{n}\right)P_{n}\mid\mathcal{F}_{n-1}\right]e_{n-1},e_{n-1}\right\rangle .
\]
Using the coercivity hypothesis yields 
\[
\mathbb{E}\left[\left\Vert e_{n}\right\Vert ^{2}\mid\mathcal{F}_{n-1}\right]\le\left(1-\gamma\right)\left\Vert e_{n-1}\right\Vert ^{2}.
\]
Iterating gives the $L^{2}$ bound. Almost sure convergence follows
as in \prettyref{thm:4-4}. 
\end{proof}

\begin{rem}
In the scalar RKHS Kaczmarz setting (see \prettyref{sec:3}), take
a reproducing kernel Hilbert space $H_{k}$ with kernel sections $k_{x}\in H_{k}$
and set 
\[
g_{n}=\frac{k_{x_{n}}}{\left\Vert k_{x_{n}}\right\Vert },\qquad P_{n}=\left|g_{n}\right\rangle \left\langle g_{n}\right|.
\]
Given data $y_{n}$, define the intermittent update 
\[
f_{n+1}=f_{n}+\lambda\eta_{n}\frac{y_{n}-f_{n}\left(x_{n}\right)}{\left\Vert k_{x_{n}}\right\Vert ^{2}}k_{x_{n}}.
\]
If $f^{*}$ is a target function and we write $e_{n}=f_{n}-f^{*}$,
then in the noiseless interpolation case $y_{n}=f^{*}\left(x_{n}\right)$
we have 
\[
e_{n+1}=\left(I-\lambda\eta_{n}P_{n}\right)e_{n},
\]
and hence the same pathwise identity holds: 
\[
\left\Vert e_{n+1}\right\Vert ^{2}_{H_{k}}=\left\Vert e_{n}\right\Vert ^{2}_{H_{k}}-\eta_{n}\lambda\left(2-\lambda\right)\left|\left\langle g_{n},e_{n}\right\rangle \right|^{2}.
\]
In the noisy model $y_{n}=f^{*}\left(x_{n}\right)+\xi_{n}$, the dropout
factor $\eta_{n}$ multiplies both the dissipative term and the noise
injection term in the corresponding conditional expectation recursion,
so intermittent updates scale the bias and variance contributions
in the usual early-stopping picture in a transparent way. 
\end{rem}

\begin{cor}
Let $H_{k}$ be a reproducing kernel Hilbert space over a separable
metric space $X$, and assume that the kernel $k$ is continuous,
so that every $f\in H_{k}$ is continuous. Let $\rho$ be a Borel
probability measure on $X$, and assume that $\left(x_{n}\right)_{n\ge1}$
are i.i.d.\ with law $\rho$. Assume that $\eta_{n}$ are i.i.d.\ $\left\{ 0,1\right\} $-valued
with $\mathbb{P}\left(\eta_{n}=1\right)=p\in\left(0,1\right)$, independent
of $\left(x_{n}\right)$. Define 
\[
g_{n}=\frac{k_{x_{n}}}{\left\Vert k_{x_{n}}\right\Vert },\qquad P_{n}=\left|g_{n}\right\rangle \left\langle g_{n}\right|.
\]
Let $S=\mathrm{supp}\left(\rho\right)$ and let 
\[
M=\left\{ f\in H_{k}:\ f\left(x\right)=0\ \text{for all }x\in S\right\} .
\]
Suppose that there exists $\alpha>0$ such that for every $f\in M^{\perp}$,
\[
\int\frac{\left|f\left(x\right)\right|^{2}}{\left\Vert k_{x}\right\Vert ^{2}}d\rho\left(x\right)\ge\alpha\left\Vert f\right\Vert ^{2}_{H_{k}}.
\]
Let $f^{*}\in H_{k}$ and consider the noiseless interpolation model
$y_{n}=f^{*}\left(x_{n}\right)$ with intermittent update 
\[
f_{n+1}=f_{n}+\lambda\eta_{n}\frac{y_{n}-f_{n}\left(x_{n}\right)}{\left\Vert k_{x_{n}}\right\Vert ^{2}}k_{x_{n}},\qquad0<\lambda<2.
\]
Then, assuming $f_{0}\in M^{\perp}$ (e.g.\ $f_{0}=0$), we have
\[
\mathbb{E}\left[\left\Vert f_{n}-P_{M^{\perp}}f^{*}\right\Vert ^{2}_{H_{k}}\right]\le\left(1-p\alpha\lambda\left(2-\lambda\right)\right)^{n}\left\Vert f_{0}-P_{M^{\perp}}f^{*}\right\Vert ^{2}_{H_{k}},
\]
so $f_{n}\to P_{M^{\perp}}f^{*}$ in $L^{2}$. 
\end{cor}

\begin{proof}
For each $n$, the kernel section satisfies $\left\langle k_{x_{n}},f\right\rangle =f\left(x_{n}\right)$,
hence 
\[
P_{n}f=\left\langle g_{n},f\right\rangle g_{n}=\frac{f\left(x_{n}\right)}{\left\Vert k_{x_{n}}\right\Vert ^{2}}k_{x_{n}}.
\]
In particular $\ker P_{n}=\left\{ f:\ f\left(x_{n}\right)=0\right\} $.
Since $k$ is continuous and $X$ is separable, the i.i.d.\ sample
$\left\{ x_{n}\right\} $ is almost surely dense in $S=\mathrm{supp}\left(\rho\right)$.
Therefore, if $f\in\bigcap_{n\ge1}\ker P_{n}$ then $f\left(x_{n}\right)=0$
for all $n$, and by continuity $f$ vanishes on the closure of $\left\{ x_{n}\right\} $,
hence on $S$. This shows that almost surely 
\[
\bigcap_{n\ge1}\ker P_{n}=M.
\]
Let $\mathcal{F}_{n}$ be the natural filtration generated by $\left\{ \left(x_{j},\eta_{j}\right)\right\} _{1\le j\le n}$.
Since $\left(x_{n}\right)$ are i.i.d., $P_{n}$ is independent of
$\mathcal{F}_{n-1}$ and has the same law as $P_{1}$, so 
\[
\mathbb{E}\left[P_{n}\mid\mathcal{F}_{n-1}\right]=\mathbb{E}\left[P_{1}\right]
\]
for all $n$. Moreover, $\eta_{n}$ is independent of $P_{n}$ and
of $\mathcal{F}_{n-1}$, hence $p_{n}=\mathbb{E}\left[\eta_{n}\mid\mathcal{F}_{n-1}\right]=p$.
A direct computation gives, for $f\in H_{k}$, 
\[
\left\langle \mathbb{E}\left[P_{1}\right]f,f\right\rangle =\mathbb{E}\left[\left|\left\langle g_{1},f\right\rangle \right|^{2}\right]=\int\frac{\left|f\left(x\right)\right|^{2}}{\left\Vert k_{x}\right\Vert ^{2}}d\rho\left(x\right).
\]
Thus the hypothesis on $\rho$ is exactly the coercivity condition
in \prettyref{thm:4-4} with $\beta=p\alpha$. Applying \prettyref{thm:4-4}
to the error $e_{n}=f_{n}-f^{*}$ gives $e_{n}\to P_{M}e_{0}=P_{M}f_{0}-P_{M}f^{*}$,
hence 
\[
f_{n}=f^{*}+e_{n}\to P_{M}f_{0}+P_{M^{\perp}}f^{*}.
\]
Under the assumption $f_{0}\in M^{\perp}$ one has $P_{M}f_{0}=0$,
which is the stated convergence to $P_{M^{\perp}}f^{*}$ in $L^{2}$.
\end{proof}

\section{Kernel reparametrization from tree-splitting}\label{sec:5}

We keep the multichannel tree-splitting setting of \prettyref{subsec:2-4}.
In particular, for each depth $N\ge1$, the splitting identity \eqref{eq:2-4}
yields a canonical coordinate map $W_{N}$ into a Hilbert direct sum
$H_{N}$, where 
\[
H_{N}:=\left(\bigoplus\nolimits_{\left|\alpha\right|=N}H\right)\oplus\left(\bigoplus\nolimits_{\left|\alpha\right|<N}H\right),
\]
equipped with the standard $\ell^{2}$ direct-sum inner product. 

We then set 
\begin{equation}
W_{N}x:=\left(\left(A_{\alpha}x\right)_{\left|\alpha\right|=N},\left(DA_{\alpha}x\right)_{\left|\alpha\right|<N}\right),\quad x\in H.\label{eq:5-1}
\end{equation}
With this choice of norm on $H_{N}$, for all $x\in H$, 
\begin{equation}
\left\Vert W_{N}x\right\Vert ^{2}=\sum_{\left|\alpha\right|=N}\left\Vert A_{\alpha}x\right\Vert ^{2}+\sum_{\left|\alpha\right|<N}\left\Vert DA_{\alpha}x\right\Vert ^{2}=\left\Vert x\right\Vert ^{2}.\label{eq:5-2}
\end{equation}
Consequently $W_{N}$ is an isometry and therefore, for all $x,y\in H$,
\begin{equation}
\left\langle x,y\right\rangle =\left\langle W_{N}x,W_{N}y\right\rangle .\label{eq:5-3}
\end{equation}
This observation \eqref{eq:5-2}-\eqref{eq:5-3} is the starting point
behind the constructions below.

Let $X$ be any set. Let $k_{0}$ be a positive definite kernel on
$X\times X$ with feature map $\Phi_{0}\colon X\to H$, so that 
\begin{equation}
k_{0}\left(x,y\right):=\left\langle \Phi_{0}\left(x\right),\Phi_{0}\left(y\right)\right\rangle .\label{eq:5-4}
\end{equation}
For each $N\ge1$, we form the multiscale feature map $\Phi_{N}:=W_{N}\circ\Phi_{0}$,
i.e., composition \eqref{eq:5-1} and $\Phi_{0}$. Since $W_{N}$
is an isometry, and using \eqref{eq:5-4} we obtain the identity 
\begin{equation}
k_{0}\left(x,y\right)=\left\langle \Phi_{N}\left(x\right),\Phi_{N}\left(y\right)\right\rangle ,\quad x,y\in X.
\end{equation}
Thus the tree coordinates give a structured representation of the
same kernel induced by $\Phi_{0}$, without modifying $k_{0}$.

\subsection*{Kernel truncation}

A family of new kernels arises once we discard (or weight) some of
the coordinates of $W_{N}$. We give one basic truncation that will
be sufficient for later use. 

Fix integers $0\le M\le N$, and define the truncated feature map
by keeping only the defect blocks up to depth $M-1$, 
\begin{equation}
\widetilde{\Phi}_{M}\left(x\right):=\left(\left(DA_{\alpha}\Phi_{0}\left(x\right)\right)_{\left|\alpha\right|<M}\right),\quad x\in X,\label{eq:5-6}
\end{equation}
viewed as an element of the corresponding direct sum of copies of
$H$. We define the associated truncated kernel by 
\begin{align}
\widetilde{k}_{M}\left(x,y\right) & :=\left\langle \widetilde{\Phi}_{M}\left(x\right),\widetilde{\Phi}_{M}\left(y\right)\right\rangle \nonumber \\
 & =\sum_{\left|\alpha\right|<M}\left\langle DA_{\alpha}\Phi_{0}\left(x\right),DA_{\alpha}\Phi_{0}\left(y\right)\right\rangle .
\end{align}
Since $\widetilde{k}_{M}$ is a Gram kernel of $\widetilde{\Phi}_{M}$,
it is positive definite.

To quantify the difference between $k_{0}$ and $\widetilde{k}_{M}$,
we introduce the discarded energy at depth $N$ beyond the truncation
level $M$, 
\begin{equation}
R_{N,M}\left(x\right):=\sum_{\left|\alpha\right|=N}\left\Vert A_{\alpha}\Phi_{0}\left(x\right)\right\Vert ^{2}+\sum_{M\le\left|\alpha\right|<N}\left\Vert DA_{\alpha}\Phi_{0}\left(x\right)\right\Vert ^{2},\quad x\in X.
\end{equation}

\begin{thm}
\label{thm:5.1}For all $x\in X$, 
\begin{equation}
k_{0}\left(x,x\right)-\widetilde{k}_{M}\left(x,x\right)=R_{N,M}\left(x\right)\ge0.\label{eq:5-9}
\end{equation}
Moreover, for all $x,y\in X$, 
\begin{equation}
\left|k_{0}\left(x,y\right)-\widetilde{k}_{M}\left(x,y\right)\right|\le\left(R_{N,M}\left(x\right)R_{N,M}\left(y\right)\right)^{1/2}.
\end{equation}
\end{thm}

\begin{proof}
Fix $N\ge1$ and $0\le M\le N$. Recall that 
\[
W_{N}\Phi_{0}(x)=\left(\left(A_{\alpha}\Phi_{0}(x)\right)_{\left|\alpha\right|=N},\left(DA_{\alpha}\Phi_{0}(x)\right)_{\left|\alpha\right|<N}\right)\in H_{N},
\]
and that $\widetilde{\Phi}_{M}(x)$ from \eqref{eq:5-6} consists
of the defect blocks with $\left|\alpha\right|<M$. We view $\widetilde{\Phi}_{M}(x)$
as an element of $H_{N}$ by embedding it as the subvector of $W_{N}\Phi_{0}(x)$
supported on the coordinates $\left\{ \alpha:\left|\alpha\right|<M\right\} $
in the defect part and zero elsewhere.

Define the discarded coordinate vector by 
\[
\Psi_{N,M}(x):=\left(\left(A_{\alpha}\Phi_{0}(x)\right)_{\left|\alpha\right|=N},\left(DA_{\alpha}\Phi_{0}(x)\right)_{M\le\left|\alpha\right|<N}\right)\in H_{N}.
\]
In the direct sum $H_{N}$, different coordinate blocks are mutually
orthogonal, so vectors supported on disjoint index sets are orthogonal.
By construction, the supports of $\widetilde{\Phi}_{M}(x)$ and $\Psi_{N,M}(x)$
in $H_{N}$ are disjoint; hence we have the orthogonal decomposition
\[
W_{N}\Phi_{0}(x)=\widetilde{\Phi}_{M}(x)\oplus\Psi_{N,M}(x).
\]

For the diagonal identity, take norms and use Pythagoras: 
\[
\left\Vert W_{N}\Phi_{0}(x)\right\Vert ^{2}=\Vert\widetilde{\Phi}_{M}(x)\Vert^{2}+\Vert\Psi_{N,M}(x)\Vert^{2}.
\]
Since $W_{N}$ is an isometry, $\left\Vert W_{N}\Phi_{0}(x)\right\Vert ^{2}=\left\Vert \Phi_{0}(x)\right\Vert ^{2}=k_{0}(x,x)$,
while by definition $\Vert\widetilde{\Phi}_{M}(x)\Vert^{2}=\widetilde{k}_{M}(x,x)$.
Finally, expanding the direct-sum norm of $\Psi_{N,M}(x)$ gives 
\[
\left\Vert \Psi_{N,M}(x)\right\Vert ^{2}=\sum_{\left|\alpha\right|=N}\left\Vert A_{\alpha}\Phi_{0}(x)\right\Vert ^{2}+\sum_{M\le\left|\alpha\right|<N}\left\Vert DA_{\alpha}\Phi_{0}(x)\right\Vert ^{2}=R_{N,M}(x).
\]
Substituting these identities yields \eqref{eq:5-9}, and in particular
$R_{N,M}(x)\ge0$.

For the off-diagonal bound, write the inner product using the same
decomposition: 
\begin{align*}
k_{0}(x,y) & =\left\langle \Phi_{0}(x),\Phi_{0}(y)\right\rangle =\left\langle W_{N}\Phi_{0}(x),W_{N}\Phi_{0}(y)\right\rangle \\
 & =\left\langle \widetilde{\Phi}_{M}(x),\widetilde{\Phi}_{M}(y)\right\rangle +\left\langle \Psi_{N,M}(x),\Psi_{N,M}(y)\right\rangle \\
 & =\widetilde{k}_{M}(x,y)+\left\langle \Psi_{N,M}(x),\Psi_{N,M}(y)\right\rangle .
\end{align*}
Hence 
\[
k_{0}(x,y)-\widetilde{k}_{M}(x,y)=\left\langle \Psi_{N,M}(x),\Psi_{N,M}(y)\right\rangle ,
\]
and Cauchy-Schwarz in $H_{N}$ gives 
\[
\left|k_{0}(x,y)-\widetilde{k}_{M}(x,y)\right|\le\left\Vert \Psi_{N,M}(x)\right\Vert \left\Vert \Psi_{N,M}(y)\right\Vert =\left(R_{N,M}(x)R_{N,M}(y)\right)^{1/2}.
\]
\end{proof}

In particular, any estimate on the tree energies from the previous
section (for example, the geometric decay conclusion of \prettyref{cor:2-6})
immediately yields quantitative control of $R_{N,M}\left(x\right)$,
and hence of $k_{0}-\widetilde{k}_{M}$, without additional arguments. 
\begin{cor}
\label{cor:5-2} Assume there exists a constant $c\in(0,1)$ such
that 
\[
\sum^{d}_{i=1}A^{*}_{i}A_{i}\le cI.
\]
Then for every $M\ge0$, every $N\ge M$, and every $x\in X$, one
has 
\begin{equation}
R_{N,M}\left(x\right)=\sum_{\left|\alpha\right|=M}\left\Vert A_{\alpha}\Phi_{0}\left(x\right)\right\Vert ^{2}\le c^{M}k_{0}\left(x,x\right).\label{eq:5-9-1}
\end{equation}
Consequently, 
\begin{equation}
0\le k_{0}\left(x,x\right)-\widetilde{k}_{M}\left(x,x\right)\le c^{M}k_{0}\left(x,x\right),\label{eq:5-12}
\end{equation}
and for all $x,y\in X$, 
\[
\left|k_{0}\left(x,y\right)-\widetilde{k}_{M}\left(x,y\right)\right|\le c^{M}\left(k_{0}\left(x,x\right)k_{0}\left(y,y\right)\right)^{1/2}.
\]
\end{cor}

\begin{proof}
Fix $x\in X$ and set $v=\Phi_{0}\left(x\right)\in H$. In the notation
of \prettyref{def:2-6}, the level energies satisfy 
\[
L_{n}\left(v\right)=L_{n+1}\left(v\right)+\sum_{\left|\alpha\right|=n}\left\Vert DA_{\alpha}v\right\Vert ^{2}
\]
as shown in the proof of \prettyref{prop:2-3}. Summing this identity
from $n=M$ to $n=N-1$ yields 
\[
L_{M}\left(v\right)=L_{N}\left(v\right)+\sum_{M\le\left|\alpha\right|<N}\left\Vert DA_{\alpha}v\right\Vert ^{2}.
\]
By the definition of $L_{N}$ and the definition of $R_{N,M}$, the
right-hand side is 
\[
\sum_{\left|\alpha\right|=N}\left\Vert A_{\alpha}v\right\Vert ^{2}+\sum_{M\le\left|\alpha\right|<N}\left\Vert DA_{\alpha}v\right\Vert ^{2}=R_{N,M}\left(x\right).
\]
Hence $R_{N,M}\left(x\right)=L_{M}\left(v\right)=\sum_{\left|\alpha\right|=M}\left\Vert A_{\alpha}v\right\Vert ^{2}$,
proving the identity in \eqref{eq:5-9-1}. Under the uniform contraction
hypothesis, \prettyref{cor:2-6} gives $L_{M}\left(v\right)\le c^{M}\left\Vert v\right\Vert ^{2}$.
\end{proof}

We also include a sample level inequality for Gram matrices, which
follows from the same feature representation. Given points $x_{1},\dots,x_{m}\in X$,
let $K_{0}$ and $\widetilde{K}_{M}$ be the Gram matrices 
\[
\left(K_{0}\right)_{j\ell}=k_{0}\left(x_{j},x_{\ell}\right),\qquad(\widetilde{K}_{M})_{j\ell}=\widetilde{k}_{M}\left(x_{j},x_{\ell}\right).
\]

\begin{cor}
One has $0\leq K_{0}-\widetilde{K}_{M}$, 
\begin{align}
\mathrm{tr}\left(K_{0}-\widetilde{K}_{M}\right) & =\sum\nolimits^{m}_{j=1}R_{N,M}\left(x_{j}\right),\\
K_{0}-\widetilde{K}_{M} & \leq\left(\sum\nolimits^{m}_{j=1}R_{N,M}\left(x_{j}\right)\right)I.\label{eq:5-14}
\end{align}
\end{cor}

\begin{proof}
$K_{0}-\widetilde{K}_{M}$ is the Gram matrix of the discarded coordinates,
hence is positive semidefinite. The trace identity follows from \eqref{eq:5-9}
above. The final bound is the elementary implication $0\leq M\leq\mathrm{tr}\left(M\right)I$
for positive semidefinite matrices $M$. 
\end{proof}

\begin{rem}
The construction is useful even though it does not change the base
kernel $k_{0}$ until one truncates. In many situations one wants
to work with a simpler or cheaper feature map than $\Phi_{0}$, but
still remain close to the behaviour of the original kernel. The multiscale
maps $\Phi_{N}$ and their truncations $\widetilde{\Phi}_{M}$ organize
the information in $\Phi_{0}$ by depth, and the discarded quantity
$R_{N,M}\left(x\right)$ measures exactly how much of $\Phi_{0}\left(x\right)$
has been dropped at level $N$ beyond depth $M$. The identities in
this section show that $R_{N,M}\left(x\right)$ is the difference
$k_{0}\left(x,x\right)-\widetilde{k}_{M}\left(x,x\right)$, and that
it bounds $|k_{0}\left(x,y\right)-\widetilde{k}_{M}\left(x,y\right)|$
for all $x,y$. On a finite sample $\left\{ x_{j}\right\} $, the
Gram matrices satisfy \eqref{eq:5-14}, so the eigenvalues of $\widetilde{K}_{M}$
stay uniformly close to those of $K_{0}$ whenever the discarded energies
$R_{N,M}\left(x_{j}\right)$ are small. In practice this gives a principled
way to choose the truncation depth $M$ (or to choose which nodes
to keep): one can aim for a reduction in complexity while still keeping
quantitative control on how far the truncated kernel is from the original
one. 
\end{rem}

\section{Greedy refinement and kernel compression}\label{sec:6}

We continue in the setting of \prettyref{sec:5}. Fix a depth $N\ge1$
once and for all. Write 
\begin{equation}
\mathcal{A}_{<N}:=\left\{ \alpha:\left|\alpha\right|<N\right\} 
\end{equation}
for the set of nodes strictly above level $N$ (including the empty
word $\emptyset$), and for each $\alpha\in\mathcal{A}_{<N}$ write
$K_{\alpha}$ for the Gram kernel induced by the $\alpha$-defect
block, 
\[
k_{\alpha}\left(x,y\right):=\left\langle DA_{\alpha}\Phi_{0}\left(x\right),DA_{\alpha}\Phi_{0}\left(y\right)\right\rangle ,\quad x,y\in X.
\]
Thus $k_{\alpha}$ is positive definite and for every finite sample
$\left\{ x_{1},\dots,x_{m}\right\} $ its Gram matrix 
\[
\left(K_{\alpha}\right)_{j\ell}=k_{\alpha}\left(x_{j},x_{\ell}\right)
\]
is positive semidefinite.

As in \prettyref{sec:5}, the depth-$M$ truncation corresponds to
keeping the defect blocks on the prefix set $\left\{ \alpha:\left|\alpha\right|<M\right\} $.
For adaptive refinement it is convenient to allow more general prefix
sets.
\begin{defn}
A subset $S\subseteq\mathcal{A}_{<N}$ is called \emph{prefix closed}
if $\alpha\in S$ and $\alpha\neq\emptyset$ implies $\alpha^{-}\in S$,
where $\alpha^{-}$ denotes the parent of $\alpha$ (delete the last
symbol).
\end{defn}

Given such an $S$, define the truncated feature map 
\begin{equation}
\Phi_{S}\left(x\right):=\left(\left(DA_{\alpha}\Phi_{0}\left(x\right)\right)_{\alpha\in S}\right),\qquad x\in X,
\end{equation}
viewed in the Hilbert direct sum of copies of $H$ indexed by $S$,
and define its Gram kernel 
\begin{align}
k_{S}\left(x,y\right) & :=\left\langle \Phi_{S}\left(x\right),\Phi_{S}\left(y\right)\right\rangle \nonumber \\
 & =\sum_{\alpha\in S}\left\langle DA_{\alpha}\Phi_{0}\left(x\right),DA_{\alpha}\Phi_{0}\left(y\right)\right\rangle .
\end{align}

Since $k_{S}$ is a Gram kernel, it is positive definite. The residual
energy beyond $S$ at depth $N$ is defined by 
\begin{equation}
R_{N,S}\left(x\right):=\sum_{\left|\alpha\right|=N}\left\Vert A_{\alpha}\Phi_{0}\left(x\right)\right\Vert ^{2}+\sum_{\alpha\in\mathcal{A}_{<N}\setminus S}\left\Vert DA_{\alpha}\Phi_{0}\left(x\right)\right\Vert ^{2},\qquad x\in X.
\end{equation}
When $S=\left\{ \alpha:\left|\alpha\right|<M\right\} $ this is exactly
the quantity $R_{N,M}$ from \prettyref{sec:5}. The same orthogonal
splitting argument as in \prettyref{sec:5} gives the following extension.
\begin{cor}
\label{cor:6-1} For all $x\in X$ one has 
\begin{equation}
k_{0}\left(x,x\right)-k_{S}\left(x,x\right)=R_{N,S}\left(x\right)\geq0.
\end{equation}
Moreover, for all $x,y\in X$, 
\begin{equation}
\left|k_{0}\left(x,y\right)-k_{S}\left(x,y\right)\right|\le\left(R_{N,S}\left(x\right)R_{N,S}\left(y\right)\right)^{1/2}.
\end{equation}
\end{cor}

On a finite sample $\left\{ x_{1},\dots,x_{m}\right\} $ we denote
by $K_{S}$ the Gram matrix of $k_{S}$ and by $K_{0}$ the Gram matrix
of $k_{0}$. Then $K_{S}=\sum_{\alpha\in S}K_{\alpha}$ and, exactly
as before, $K_{0}-K_{S}$ is the Gram matrix of the discarded coordinates
and hence positive semidefinite. In particular, 
\begin{gather*}
0\leq K_{0}-K_{S}\leq\mathrm{tr}\left(K_{0}-K_{S}\right)I\\
\mathrm{tr}\left(K_{0}-K_{S}\right)=\sum^{m}_{j=1}R_{N,S}\left(x_{j}\right).
\end{gather*}
This trace quantity will be the main error measure in the greedy compression
scheme below.

We next describe a simple refinement rule on the tree. Fix a sample
$\left\{ x_{1},\dots,x_{m}\right\} \subset X$. For each node $\alpha\in\mathcal{A}_{<N}$
define its sample energy 
\[
E\left(\alpha\right):=\mathrm{tr}\left(K_{\alpha}\right)=\sum^{m}_{j=1}\left\Vert DA_{\alpha}\Phi_{0}\left(x_{j}\right)\right\Vert ^{2}.
\]
If $S\subseteq\mathcal{A}_{<N}$ is prefix closed, its \emph{frontier}
is 
\[
\partial S:=\left\{ \beta\in\mathcal{A}_{<N}\setminus S:\beta^{-}\in S\right\} .
\]
The greedy rule is:
\begin{lyxalgorithm*}
\emph{Start with a prefix closed set $S_{0}$ (for instance $S_{0}=\left\{ \emptyset\right\} $).
Given $S_{t}$, choose} 
\[
\alpha_{t}\in\arg\max_{\beta\in\partial S_{t}}E\left(\beta\right),
\]
\emph{and set $S_{t+1}:=S_{t}\cup\left\{ \alpha_{t}\right\} $.} 
\end{lyxalgorithm*}
\noindent Each step adds one defect block and preserves prefix closedness.
The next proposition gives the basic monotonicity property and identifies
the exact decrease of the trace error.
\begin{prop}
\label{prop:6-2} For every $t$ one has 
\[
K_{S_{t}}\leq K_{S_{t+1}}\leq K_{0},
\]
and 
\[
\mathrm{tr}\left(K_{0}-K_{S_{t+1}}\right)=\mathrm{tr}\left(K_{0}-K_{S_{t}}\right)-E\left(\alpha_{t}\right).
\]
In particular, $\mathrm{tr}\left(K_{0}-K_{S_{t}}\right)$ is nonincreasing
in $t$ and provides an a posteriori stopping criterion. 
\end{prop}

\begin{proof}
Since $K_{S_{t+1}}=K_{S_{t}}+K_{\alpha_{t}}$ and $K_{\alpha_{t}}\geq0$,
one has $K_{S_{t}}\leq K_{S_{t+1}}$. Also $K_{0}-K_{S_{t}}$ is the
Gram matrix of discarded coordinates and hence is positive semidefinite,
so $K_{S_{t}}\leq K_{0}$. Taking traces gives 
\[
\mathrm{tr}\left(K_{0}-K_{S_{t+1}}\right)=\mathrm{tr}\left(K_{0}-K_{S_{t}}-K_{\alpha_{t}}\right)=\mathrm{tr}\left(K_{0}-K_{S_{t}}\right)-\mathrm{tr}\left(K_{\alpha_{t}}\right),
\]
and $\mathrm{tr}\left(K_{\alpha_{t}}\right)=E\left(\alpha_{t}\right)$
by definition. 
\end{proof}

\subsection*{Kernel learning}

We next discuss the consequence most directly aligned with standard
kernel learning. Let $y\in\mathbb{R}^{m}$ be labels on the sample
and fix $\lambda>0$. For a Gram matrix $K$ define the kernel ridge
predictor on the sample by 
\[
\widehat{y}\left(K\right):=K\left(K+\lambda I\right)^{-1}y\in\mathbb{R}^{m}.
\]
The next proposition shows that the trace bound controls the error
incurred by replacing the full kernel by a greedy truncation.
\begin{thm}
\label{thm:6-3} For every prefix closed $S\subseteq\mathcal{A}_{<N}$
one has 
\begin{equation}
\left\Vert \widehat{y}\left(K_{0}\right)-\widehat{y}\left(K_{S}\right)\right\Vert \le\frac{1}{\lambda}\left\Vert K_{0}-K_{S}\right\Vert \left\Vert y\right\Vert \le\frac{1}{\lambda}\mathrm{tr}\left(K_{0}-K_{S}\right)\left\Vert y\right\Vert .
\end{equation}
In particular, along the greedy sequence $\left(S_{t}\right)$ one
has 
\begin{equation}
\left\Vert \widehat{y}\left(K_{0}\right)-\widehat{y}\left(K_{S_{t}}\right)\right\Vert \le\frac{1}{\lambda}\mathrm{tr}\left(K_{0}-K_{S_{t}}\right)\left\Vert y\right\Vert =\frac{1}{\lambda}\sum^{m}_{j=1}R_{N,S_{t}}\left(x_{j}\right)\left\Vert y\right\Vert .
\end{equation}
\end{thm}

\begin{proof}
Using $\widehat{y}\left(K\right)=(I-\lambda\left(K+\lambda I\right)^{-1})y$,
one obtains 
\[
\widehat{y}\left(K_{0}\right)-\widehat{y}\left(K_{S}\right)=\lambda\left(\left(K_{S}+\lambda I\right)^{-1}-\left(K_{0}+\lambda I\right)^{-1}\right)y.
\]
The resolvent identity gives 
\[
\left(K_{S}+\lambda I\right)^{-1}-\left(K_{0}+\lambda I\right)^{-1}=\left(K_{S}+\lambda I\right)^{-1}\left(K_{0}-K_{S}\right)\left(K_{0}+\lambda I\right)^{-1}.
\]
Since $K_{0}\geq0$ and $K_{S}\geq0$, one has 
\[
\Vert\left(K_{0}+\lambda I\right)^{-1}\Vert\le\frac{1}{\lambda},\qquad\Vert\left(K_{S}+\lambda I\right)^{-1}\Vert\le\frac{1}{\lambda},
\]
hence 
\[
\left\Vert \widehat{y}\left(K_{0}\right)-\widehat{y}\left(K_{S}\right)\right\Vert \le\lambda\cdot\frac{1}{\lambda}\left\Vert K_{0}-K_{S}\right\Vert \cdot\frac{1}{\lambda}\left\Vert y\right\Vert =\frac{1}{\lambda}\left\Vert K_{0}-K_{S}\right\Vert \left\Vert y\right\Vert .
\]
Finally, since $K_{0}-K_{S}\geq0$, the elementary implication $0\leq M\leq\mathrm{tr}\left(M\right)I$
yields 
\[
\left\Vert K_{0}-K_{S}\right\Vert \le\mathrm{tr}\left(K_{0}-K_{S}\right),
\]
and the stated trace formula follows from the definitions of $R_{N,S}$. 
\end{proof}

\begin{rem}
\label{rem:6-4} \prettyref{thm:6-3} suggests a concrete stopping
rule. If one prescribes a tolerance $\varepsilon>0$ and chooses $t$
so that 
\[
\sum^{m}_{j=1}R_{N,S_{t}}\left(x_{j}\right)\le\varepsilon\lambda,
\]
then the sample-level ridge predictor satisfies 
\[
\left\Vert \widehat{y}\left(K_{0}\right)-\widehat{y}\left(K_{S_{t}}\right)\right\Vert \le\varepsilon\left\Vert y\right\Vert .
\]
In this sense the discarded tree energy provides a quantitative bound
on how far the compressed kernel deviates, at the level of the induced
learning map $y\mapsto\widehat{y}\left(K\right)$. 
\end{rem}

\begin{rem}
One can consider other choices of greedy score. The scheme above is
driven by the trace bound, because it is directly computable from
the block energies $\left\Vert DA_{\alpha}\Phi_{0}\left(x_{j}\right)\right\Vert ^{2}$
and yields a clean stability estimate. In applications where one wishes
to prioritize conditioning or information gain rather than trace error,
one may instead use the same prefix-closed exploration but score candidates
$\beta\in\partial S$ by a matrix functional of the update $K_{\beta}$,
for instance through 
\[
S\mapsto\log\det\left(I+\frac{1}{\lambda}K_{S}\right).
\]
Such choices can be analyzed once one has additional structure on
the block matrices $K_{\alpha}$ on the given sample (for example,
uniformly low rank, or an efficiently maintainable factorization).
We do not follow this direction here, since the trace-based scheme
already produces an explicit, verifiable bound in terms of the discarded
energies $R_{N,S}\left(x_{j}\right)$.
\end{rem}

\section{Connections to Greedy Kernal Principal Component Analysis}\label{sec:7}

Below we give implications of our results above for Principal Component
Analysis (PCA), extending our earlier results. PCA is one of the current
and important applications of kernel analysis. We recall that PCA
offers powerful tools for dimensionality reduction, thus helping transforms
high-dimensional data into a smaller set of uncorrelated variables,
known as principal components (PCs). This is especially useful since
it can be done, retaining maximum variance (information), i.e., choices
of PCs. Thus by identifying orthogonal axes of maximum variance, PCA
allows us to 1) reduces noise, 2) handles multicollinearity, and 3)
facilitate feature selection and visualization. For the benefit of
readers we offer the following citations \cite{Belkin03,MR4616401,MR3857315,MR3883202,4106847,MR2247587,MR3005666,MR2177937,MR3854652,MR3534893}.
In recent years, the subject of kernel-principal component analysis,
and its applications, has been extensively studied, and made progress
in diverse directions \cite{MR3780557,MR3820672,MR3911884,MR1720704,MR3857315,MR3878657,MR3913046,MR3934645,MR3850675,MR3883202,MR3900802,MR3922239}. 

Here, we will tie the classical principal component analysis (PCA)
with our results in earlier sections to greedy KPCA. Please see \cite{JoSo07a,Jorgensen2022InfinitedimensionalST}
for classical PCA and \cite{pmlr-vR5-ouimet05a,WENZEL2021105508,2009SPIE.7495E..30L,DING20101542,Franc2006}
for greedy KPCA.

The word ``greedy'' in GKPCA refers to the use ``greedy'' for
design of algorithmic techniques that are used in the selection (small)
representative subset of the training data, rather than the entire
dataset, going into the particular kernel model for the problem at
hand; hence the greedy algorithms achieve a reduction of the size
of the kernel matrices that are used in computations in feature space.
Applications are many and diverse, and the list includes: Face Recognition,
Speaker Identification (nonlinear features in noisy environments),
Biomedical Signal Processing, and Classification (used in training
efficiency of classifiers like Support Vector Machines (SVM).)

Above \prettyref{sec:3} gives interpolation result but also a contractive
operator flow in a RKHS with exact energy bookkeeping and residual
decomposition given by rank-one projections. This is precisely the
structure underlying in principal component analysis and kernel principal
component analysis. While PCA is a special case where the projections
are chosen to maximize variance while the framework fixes the projections
and optimizes the residual energy. Our framework in \prettyref{sec:2},
\prettyref{sec:3} and \prettyref{sec:5} provide convergence conditions,
residual energy decomposition and stability under noise \prettyref{prop:3-2}
which is stronger than classical PCA/KPCA \cite{pmlr-vR5-ouimet05a,WENZEL2021105508,2009SPIE.7495E..30L,DING20101542,Franc2006}.

We begin by defining the greedy kernel principal component analysis
(KPCA).
\begin{defn}
\cite{pmlr-vR5-ouimet05a} Let $X=\{x_{1}\cdots,x_{m}\}$ be a training
set in input space. Let $H_{k}$ be an RKHS of complex-valued functions
on $X$ with reproducing kernel $k$. For $x_{n}\in X$ the kernel
section at $x_{n}$ is $k_{x_{n}}:=k\left(\cdot,x_{n}\right)\in H_{k}$,
where $g_{n}:=\frac{k_{x_{n}}}{\sqrt{k\left(x_{n},x_{n}\right)}}$
is the normalized kernel feature. The rank-one RKHS projection is
defined by 
\[
P_{n}f=\frac{f\left(x_{n}\right)}{k\left(x_{n},x_{n}\right)}k_{x_{n}}.
\]
\end{defn}

\begin{lyxalgorithm}
\cite{pmlr-vR5-ouimet05a} (Greedy Subset Selection) The greedy KPCA
builds a dictionary $D=\{d_{1}\cdots,d_{l}\}\subseteq X$ of training
points selected greedily based on how well a data model is approximated
by the span of already-selected dictionary elements.

At iteration $n$, with dictionary $D_{n-1}=\{d_{1}\cdots,d_{n-1}\}$,
compute the following two for the next training point $x_{n}$:

1. The feature-space projection of of $k_{x_{n}}$, onto the span
of dictionary features $\{\phi(d_{i})\}$,

2. The residual norm in feature space:
\[
\delta_{n}:=\min_{a\in R^{n-1}}\Vert\phi(x_{n})-\sum^{n-1}_{i=1}a_{i}\phi(d_{i})\Vert^{2}
\]
where $\phi(x_{n})=k_{x_{n}}$is the RKHS feature map. 

The greedy rule selects the next training point $x_{n}$ into the
dictionary if $\delta_{n}>\epsilon$ , then $d_{n}=x_{n}$, otherwise
$x_{n}$ is not added thus skipped.
\end{lyxalgorithm}

This criterion greedily increases the span in RKHS that best approximates
the full dataset’s feature map, choosing points that are most distant
in residual norm from the current subspace in feature space.
\begin{lyxalgorithm}
\cite{pmlr-vR5-ouimet05a} (Greedy KPCA Subspace Approximation) From
the dictionary $D=\{d_{1}\cdots,d_{n}\}\subseteq X$ built above:

1. Form the dictionary Gram matrix $G\in\mathbb{R}^{n\times n}$ with
\[
G_{ij}=k(d_{j},d_{j}).
\]

2. Compute the eigenvalues $\lambda_{1}\geq\cdots\geq\lambda_{p}$
of the Gram matrix, $G$ and the corresponding eigenvectors $v_{1},\cdots,v_{p}$
.

3. Construct the approximate KPCA components:
\[
u^{(D)}_{j}=\sum^{n}_{i=1}\left(v_{j}\right)_{i}k_{d_{i}},\quad where\quad j=1,\cdots,p.
\]

Here, $u^{(D)}_{j}\in H_{k}$ are approximate principal directions
in RKHS associated with the dictionary $D=\{d_{1}\cdots,d_{n}\}$,
and they approximate the true KPCA eigenfunctions that would be obtained
from the full Gram matrix on $X$.
\end{lyxalgorithm}

\begin{defn}
\cite{pmlr-vR5-ouimet05a} With the above conditions, for any new
point $x\in X$, its greedy KPCA projection/embedding onto the approximate
subspace is 
\[
f^{(D)}_{j}(x)=\langle u^{(D)}_{j},g_{x}\rangle=\sum^{n}_{i=1}\left(v_{j}\right)_{i}\frac{k(d_{i},x)}{\sqrt{k\left(x,x\right)}},\qquad for\quad j=1,\cdots,p.
\]
\end{defn}

\begin{rem}
In our framework in earlier sections, the dictionary $D$ spans a
subspace 
\[
H_{D}=\text{{span}}\{k_{d_{i}}:d_{i}\in D\}\subseteq H_{k}.
\]
 The KPCA components are the leading eigenfunctions of the empirical
covariance operator restricted to $H_{D}$. The greedy distance criterion
corresponds to selecting $x_{n}$ in order that the projection operator
onto $H_{D}$ leaves the largest residual in RKHS.
\end{rem}

We connect our theory to greedy KPCA in \cite{pmlr-vR5-ouimet05a}
which is a subset selection method for approximate eigenfunction estimation,
where dictionary points are chosen to maximize unexplained variance
in feature space. Then classical kernel PCA is applied in the reduced
dictionary span. Please see \cite{JoSo07a,Jorgensen2022InfinitedimensionalST}
for KPCA details.
\begin{rem}[Operator-level connection to PCA]
 In classical PCA in operator form, we let be a covariance operator
on $H$ and PCA yields projections:
\[
P_{n}=|g_{n}\rangle\langle g_{n}|
\]
where, $g_{n}$ are eigenvectors of $G$ \cite{JoSo07a,Jorgensen2022InfinitedimensionalST}.
In \prettyref{sec:3}, replacing static PCA with dynamic telescoping
flow
\[
T_{n}=(1-\lambda_{1}P_{1})\cdots(1-\lambda_{n}P_{n})
\]
with the identity 
\[
\left\Vert f\right\Vert ^{2}-\left\Vert T_{N}f\right\Vert ^{2}=\sum^{N}_{n=1}\lambda_{n}\left(2-\lambda_{n}\right)\left\Vert P_{n}T_{n-1}f\right\Vert ^{2}
\]
Each step removes some energy in direction $P_{n}$, the removal is
controlled and monotone, and measured exactly. Also, no eigendecomposition
is necessary.
\end{rem}

Also, in \prettyref{sec:3},
\[
g_{n}:=\frac{k_{x_{n}}}{\left\Vert k_{x_{n}}\right\Vert }
\]
updates occur in feature space, and no covariance operator is formed.
The update

\begin{equation}
f_{n}=f_{n-1}+\lambda_{n}\left(y_{n}-f_{n-1}\left(x_{n}\right)\right)\frac{k_{x_{n}}}{k\left(x_{n},x_{n}\right)},\qquad n\ge1.\label{eq:3-3-1}
\end{equation}
is equivalent to incremental KPCA

The classical KPCA lacks exact telescoping energy balance as it only
gives asymptotics or expection bounds usually. But by \prettyref{thm:3-1}
our new framework allows telescoping energy balance.

In classical PCA framework, where $y_{n}=0$, $P_{n}$ is chosen as
leading empirical variance direction and $\lambda_{n}=1$. Thus, $T_{n}=\prod^{n}_{k=1}(I-P_{k})$
which is deflation PCA. Our results provide convergence conditions,
residual energy decomposition and stability under noise by \prettyref{prop:3-2}. 

\prettyref{sec:3}, ties into exact greedy energy decomposition resulting
in the following theorem.
\begin{thm}
(Exact Greedy Energy Decomposition) Let $H_{k}$ be RKHS with kernel
$k$. Let $(x_{n})_{n\geq1}\subset X$ and let , $0<\lambda_{n}<2$.
Define the rank-one projections
\[
P_{n}f=\frac{f\left(x_{n}\right)}{k\left(x_{n},x_{n}\right)}k_{x_{n}},
\]
and the greedy residual operators
\[
T_{n}=\prod^{n}_{j=1}(I-\lambda_{j}P_{j}),\quad T_{0}=I.
\]

Then for every $f\in H_{k}$ and every $N\geq1,$ the telescoping
series\eqref{eq:3-2-1} follows
\begin{align}
\left\Vert f\right\Vert ^{2}-\left\Vert T_{N}f\right\Vert ^{2} & =\sum^{N}_{n=1}\lambda_{n}\left(2-\lambda_{n}\right)\left\Vert P_{n}T_{n-1}f\right\Vert ^{2}\nonumber \\
 & =\sum^{N}_{n=1}\lambda_{n}\left(2-\lambda_{n}\right)\frac{\left|\left(T_{n-1}f\right)\left(x_{n}\right)\right|^{2}}{k\left(x_{n},x_{n}\right)}.\label{eq:3-2-1}
\end{align}
\end{thm}

\begin{proof}
Each $P_{x_{n}}$ is an orthogonal projection, so $P_{x_{n}}=P^{*}_{x_{n}}=P^{2}_{x_{n}}$.
For any $g\in H_{k}$, 
\[
||(I-\lambda_{n}\LyXZeroWidthSpace P_{x_{n}})g||^{2}=||g||^{2}-2\lambda_{n}\LyXZeroWidthSpace||P_{x_{n}}\LyXZeroWidthSpace\LyXZeroWidthSpace g||^{2}+\lambda^{2}_{n}\LyXZeroWidthSpace||P_{x_{n}}g||^{2}=||g||^{2}-\lambda_{n}(2-\lambda_{n})||P_{x_{n}}g||^{2}.
\]
Letting $g=T_{n-1}f$ and summing from $n=1$ to $N$ yields a telescoping
series \eqref{eq:3-2-1}.
\end{proof}

This identity decomposes the total variance of $f$ into the residual
variance $\left\Vert T_{N}f\right\Vert ^{2}$ and the cumulative variance
explained at each greedy step. Unlike classical (kernel) PCA, this
decomposition holds exactly at every finite iteration and does not
rely on orthogonality or spectral assumptions.

The following theorem results from \eqref{thm:2-4} and\eqref{thm:5.1}.
A central advantage of the greedy formulation is that convergence
can be characterized purely through the step-size schedule from \eqref{thm:2-4}.
\begin{thm}
(Strong Convergence of Greedy KPCA) Let the step sizes satisfy \eqref{eq:2-3}

\begin{equation}
\sum^{\infty}_{n=1}\frac{\lambda_{n}}{2-\lambda_{n}}<\infty.\label{eq:2-3-1}
\end{equation}
Then the operator sequence $(T_{n})$ converges strongly to a bounded
operator$T_{\infty}$. In particular, for every $f\in H_{k}$,
\[
T_{n}f\to T_{\infty f}
\]
 in norm, and the residual energy $\left\Vert T_{n}f\right\Vert ^{2}$
decreases monotonically to $\left\Vert T_{\infty}f\right\Vert ^{2}$.
\end{thm}

\begin{proof}
Each factor $A_{n}=I-\lambda_{n}P_{x_{n}}$ is a contraction. Further
\[
||T_{n-1}f-T_{n}f||^{2}=\lambda^{2}_{n}||P_{x_{n}}T_{n-1}||^{2}\leq\frac{\lambda_{n}}{2-\lambda_{n}}(\left\Vert T_{n-1}f\right\Vert ^{2}-\left\Vert T_{n}f\right\Vert ^{2}),
\]
where the inequality follows from \eqref{thm:5.1}. The summability
assumption implies $(T_{n}f)$ is Cauchy for each $f$ , hence convergent
in norm. Uniform boundedness yields strong convergence of $T_{n}$
to a bounded limit $T_{\infty}$.
\end{proof}

This provides a direct analogue of full variance capture in classical
PCA, achieved without eigendecomposition. 

The following theorem is the direct result of \prettyref{prop:3-2}
\begin{thm}
(Noise Stability and Bias-Variance Decomposition) Suppose observations
were corrupted by noise added as in above \prettyref{prop:3-2},
\[
y_{n}=f^{\ast}\left(x_{n}\right)+\varepsilon_{n},
\]
with \prettyref{eq:3-5}
\[
\mathbb{E}\left(\varepsilon_{n}\mid x_{n}\right)=0,\qquad\mathbb{E}\left(\left|\varepsilon_{n}\right|^{2}\mid x_{n}\right)\le\sigma^{2}
\]
for some $\sigma\ge0$, and we assume that the noises $\left\{ \varepsilon_{n}\right\} $
are independent across $n$. Assume further that 
\[
k\left(x_{n},x_{n}\right)\ge\kappa,\qquad\sum^{\infty}_{n=1}\lambda^{2}_{n}<\infty.
\]

Then the greedy KPCA iterations satisfy the following
\[
\mathbb{E}\left\Vert f^{*}-f_{N}\right\Vert ^{2}\le\left\Vert f^{\ast}\right\Vert ^{2}-\sum^{N}_{n=1}\lambda_{n}(2-\lambda)\mathbb{\mathbb{E}}\left\Vert P_{x_{n}}e_{n-1}\right\Vert ^{2}+\sum^{N}_{n=1}\lambda^{2}_{n}\mathbb{E}\left(\frac{\left|\varepsilon_{n}\right|^{2}}{k\left(x_{n},x_{n}\right)}\right),
\]

and the noise added is uniformly bounded by 
\[
\frac{\sigma^{2}}{\kappa}\sum^{N}_{n=1}\lambda^{2}_{n}<\infty.
\]
\end{thm}

\begin{proof}
Decompose each update into a deterministic component driven by $f^{*}$
and a stochastic component drive by $\epsilon_{n}.$ The orthogonality
of $P_{x_{n}}$and the zero-mean assumption eliminate cross terms
in expectation. Now, the variance term follows from 
\[
\left\Vert \lambda_{n}P_{x_{n}}\epsilon_{n}\right\Vert ^{2}=\lambda^{2}_{n}\frac{\left|\varepsilon_{n}\right|^{2}}{k\left(x_{n},x_{n}\right)},
\]
and summability follows from the assumptions stated.
\end{proof}

This provides a transparent and quantitative criterion for robustness,
in contrast to classical kernel PCA where noise is typically handled
only indirectly through truncation or regularization.

\subsection*{Conclusion:}

Greedy Kernel PCA replaces global spectral decomposition by a contraction-driven
operator flow that incrementally removes variance along data-adaptive
directions. The merged framework of earlier sections provides an (1)
exact finite-step energy decomposition, (2) strong operator convergence
under simple step-size conditions, and (3) explicit bias--variance
tradeoffs and stability guarantees under noise.

Together, these results establish greedy kernel PCA as a principled,
online-capable alternative to classical kernel PCA, particularly well
suited for large-scale, streaming, and noise-aware learning problems. 

For various applications, please see \cite{1699563,doi:10.1137/23M1587956,4353576,2009SPIE.7495E..30L,math12132111}.

\subsection*{Acknowledgements. }

The first named author wishes to acknowledge many useful discussions
with many colleagues, including Daniel Alpay, Chad Berner, Dorin Dutkay,
John E. Herr, Halyun Jeong, and Eric S. Weber.

\bibliographystyle{amsalpha}
\bibliography{ref}

\end{document}